\documentclass[%
  onecolumn
   , hidempi
]{mpi2015-cscpreprint}

\usepackage[american]{babel}
\usepackage[utf8]{inputenc}

\usepackage{graphicx}

\usepackage{amsmath,amsfonts,amssymb,amsthm,epsfig,graphicx,url}
\usepackage{color}
\usepackage{xspace}
\usepackage{mathrsfs}
\usepackage{algorithmic}
\usepackage{algorithm}
\usepackage{bbm}
\renewcommand{\algorithmicrequire}{\textbf{Input:}}
\renewcommand{\algorithmicensure}{\textbf{Output:}}
\usepackage[colorinlistoftodos]{todonotes}
\usepackage{cite}
\usepackage{enumitem}

\usepackage{todonotes}
\usepackage{url}
\usepackage[nameinlink]{cleveref}

\newtheorem{remark}{Remark}[section]

\newtheorem{proposition}{Proposition}[section]

\def\hot#1{\textcolor{red}{#1}}

\def\IR{{\mathbb R}}
\def\IC{{\mathbb C}}

\def\IL{{\mathbb L}}
\def\IM{{\mathbb M}}

\newcommand{\bA}{{\textbf A}}
\newcommand{\bB}{{\textbf B}}
\newcommand{\bC}{{\textbf C}}

\newcommand{\bE}{{\textbf E}}
\newcommand{\bF}{{\textbf F}}
\newcommand{\bG}{{\textbf G}}
\newcommand{\bS}{{\textbf S}}
\newcommand{\bY}{{\textbf Y}}
\newcommand{\bJ}{{\textbf J}}

\newcommand{\bL}{{\textbf L}}
\newcommand{\bM}{{\textbf M}}

\newcommand{\bI}{{\textbf I}}
\newcommand{\bH}{{\textbf H}}
\newcommand{\bP}{{\textbf P}}

\newcommand{\bQ}{{\textbf Q}}
\newcommand{\bW}{{\textbf W}}

\newcommand{\bT}{{\textbf T}}
\newcommand{\bZ}{{\textbf Z}}

\newcommand{\bx}{{\textbf x}}
\newcommand{\by}{{\textbf y}}
\newcommand{\bu}{{\textbf u}}
\newcommand{\bV}{{\textbf V}}
\newcommand{\bU}{{\textbf U}}

\newcommand{\bc}{{\textbf c}}
\newcommand{\bfe}{{\textbf e}}

\newcommand{\bb}{{\textbf b}}

\newcommand{\bw}{{\textbf w}}
\newcommand{\bz}{{\textbf z}}

\newcommand{\bfz}{{\mathbf 0}}

\newcommand{\cH}{ {\cal H} }

\newcommand{\npf}{{{}\ell_p}}
\newcommand{\nqf}{{{}\ell_q}}
\newcommand{\np}{{{}N_p}}
\newcommand{\nq}{{{}N_q}}

\newcommand{\quadP}{\widetilde{\bP}}
\newcommand{\quadQ}{\widetilde{\bQ}}
\newcommand{\quadU}{\widetilde{\bU}}
\newcommand{\quadL}{\widetilde{\bL}}
\newcommand{\quadZ}{\widetilde{\bZ}}
\newcommand{\quadS}{\widetilde{\bS}}
\newcommand{\quadY}{\widetilde{\bY}}
\newcommand{\lrU}{\widehat{\bU}}
\newcommand{\lrL}{\widehat{\bL}}

\newcommand{\bbh}{\mbox{\boldmath${\mathbbm{h}}$}}
\newcommand{\bbg}{\mbox{\boldmath${\mathbbm{g}}$}}
\newcommand{\bbH}{\mbox{\boldmath${\mathbbm{H}}$}}
\newcommand{\bbG}{\mbox{\boldmath${\mathbbm{G}}$}}
\newcommand{\quadLb}{\widetilde{\bbh}}
\newcommand{\quadcU}{\widetilde{\bbg}}
\newcommand{\quadLB}{\widetilde{\bbH}}
\newcommand{\quadCU}{\widetilde{\bbG}}
\newcommand{\quadLL}{\widetilde{\IL}}
\newcommand{\quadMM}{\widetilde{\IM}}
\newcommand{\quadAr}{\widetilde{\bA}_r}
\newcommand{\quadbr}{\widetilde{\bb}_r}
\newcommand{\quadcr}{\widetilde{\bc}_r}
\newcommand{\quadEr}{\widetilde{\bE}_r}

\newcommand{\mor}{\textsf{MoR}\xspace}
\newcommand{\QBT}{\textsf{QuadBT}\xspace}
\newcommand{\BT}{\textsf{BT}\xspace}
\newcommand{\imunit}{{\dot{\imath\hspace*{-0.2em}\imath}}}

\newcommand{\lp}{\mu}
\newcommand{\rp}{\lambda}

\newcommand{\origLL}{\IL}
\newcommand{\origMM}{\IM}

\usepackage{verbatim}
\usepackage{xspace,extarrows,fdsymbol}
\renewcommand{\algorithmicrequire}{\textbf{Input:}}
\renewcommand{\algorithmicensure}{\textbf{Output:}}

\usepackage{tikz}
\usetikzlibrary{positioning}

\usepackage{pgfplots}
\pgfplotsset{compat = 1.13,
	colormap name = viridis,
	unbounded coords = jump}

\usepackage{caption}
\usepackage{subcaption}
\definecolor{myRed}{HTML}{E34A33}
\definecolor{myBlue}{HTML}{0571B0}
\definecolor{myBrown}{HTML}{A6611A}

\begin{document}
  

\title{Data-driven balancing \\ of linear dynamical systems}

\author[$\ast$]{Ion Victor Gosea}
\affil[$\ast$]{Max Planck Institute for Dynamics of Complex Technical Systems,\authorcr
	Sandtorstr. 1, 39106 Magdeburg, Germany.\authorcr
  \email{gosea@mpi-magdeburg.mpg.de}, \orcid{0000-0003-3580-4116}}
  
\author[$\dagger$]{Serkan Gugercin}
\affil[$\dagger$]{Department of Mathematics and Computational Modeling and Data
	Analytics Division, \authorcr Academy of Integrated Science, Virginia Tech, Blacksburg,
	VA 24061, USA.\authorcr
	\email{gugercin@vt.edu}, \orcid{0000-0003-4564-5999} \authorcr
	\email{beattie@vt.edu}, \orcid{0000-0003-3302-4845}}

\author[$\dagger$]{Christopher Beattie}

\shorttitle{Quadrature-based balanced truncation from data}
\shortauthor{I. V. Gosea, S. Gugercin, C. Beattie}
\shortdate{}

\keywords{Balanced truncation, numerical quadrature, data-driven modeling, transfer functions, impulse responses,  Gramians.}

\msc{37M99, 41A20, 65D30, 65K99, 93A15, 93B15.}

\abstract{%
	We present a novel reformulation of balanced truncation, a classical model reduction method. The principal innovation that we introduce comes through the use of system response data that has been either measured or computed, without reference to any prescribed realization of the original model. Data are represented by sampled values of the transfer function {or the impulse response}  corresponding to the original model. We discuss parallels that our approach bears with the Loewner framework, another popular data-driven method. We illustrate our approach numerically in both continuous-time and discrete-time cases. 
}

\maketitle


\section{Introduction} 
\label{intro}

Model reduction (\mor) here refers to 
system-theoretic techniques used to create compactly represented \emph{reduced models} that are capable of reproducing the input-output behaviour of large-scale dynamical systems with high accuracy. 
This is accomplished by encoding fine scale dynamical features of the original systems efficiently into \emph{reduced} dynamical systems, allowing them to closely  mimic the input/output behavior of the original system for a wide range of input conditions, while having substantially lower order than the system whose behaviour they mimic.  Such reduced models should be cheap to simulate and easy to manipulate and control.
We refer the reader to \cite{ACA05, AntBG20,siammorbook2017, quarteroni2015reduced} for details on a variety of \mor techniques. 

Balanced truncation (\BT) \cite{moore1981principal,mullis1976synthesis} is one of the most successful and commonly deployed \mor methods. For linear dynamical systems such as those that we consider in this paper, \BT retains asymptotic stability of the original systems and provides \emph{a priori} bounds for the model reduction error. \BT effectively discards those states that are both difficult to reach and to observe, as quantified through the relative magnitude of the system's Hankel singular values. Typically, the loss of these states is expected to have little effect on the observed input-output dynamics of the system and the resulting reduced model may replace the original system in simulations or analysis without much lost with regard to accuracy. Extensions of \BT to nonlinear systems have emerged over the years, starting with \cite{morSch93} and continuing with more recent approaches tailored to particular classes of nonlinearities, such as bilinear systems in \cite{damm11}, quadratic-bilinear systems in \cite{morBTQBgoyal}, and to switched systems in \cite{morGosPAetal18}.

The major computational cost of \BT is the need to solve large-scale Lyapunov equations in order to obtain the system Gramians, or better, their square-root factors. Many methods have been proposed to accomplish this efficiently; see, e.g., \cite{WP02,li2002low,penzl1998numerical,BenS13,simoncini2016computational,kurschner2016efficient,sabino2007solution} and references therein. All such methods make central use of the system state-space representation and, in the context of \BT, 
produce explicit state-space projections that ultimately are engaged to generate a reduced model. 
Thus, \BT is ``intrusive" to the extent that internal representations of the system dynamics play a central role in determining a final reduced model.  
This is to be contrasted with ``nonintrusive" methods that are \emph{data-driven}, requiring access only to system response and behavior data, e.g., transfer function or impulse response measurements. 

In this paper, we develop a new, data-driven formulation for \BT. We achieve this by recognizing that \BT does not make independent use of the two system Gramians, which each depend on internal and hence generally inaccessible variables.   \BT rather makes fundamental use of their \emph{product}, which preserves system invariants that do not depend on an underlying system realization.  In the present work, we describe how this product of Gramians can be approximated directly through transfer function sampling to any accuracy desired.  We explicitly derive reduced-order quantities using observed response data and as consequence, arrive at a novel, \emph{non-intrusive} formulation of \BT.

The rest of the paper is organized as follows: 
In \S \ref{sec:BT}, we review the basic elements of \BT and detail the usual steps involved for computing balanced ROMs. We introduce our main result in \S \ref{sec:QuadBT}, a data-driven approach to \BT and discuss how one estimates key quantities through quadrature in the  frequency-domain, that in turn, involves quantities that are extracted from data. 
In \S \ref{sec:ref_and_ext}, we include refinements that take into account symmetries implicit in real-valued dynamics as well as extensions to MIMO, discrete-time {and application to infinite-dimensional} systems. Even though our main focus is frequency domain data, in \S \ref{sec:timeQBT}  we revisit the problem using time-domain data instead and make connections to earlier works. Conclusions are provided in \S \ref{sec:conc}, while appendices elaborating on the quadrature rules considered and 
detailing the proof of Proposition \ref{prop:QuadError} follow the bibliography.  Numerical experiments and various illustrations are provided in \S\S \ref{sec:QuadBT}, \ref{sec:ref_and_ext}, and \ref{sec:timeQBT}.

\section{Balanced Truncation}  \label{sec:BT}

Consider the linear time-invariant (LTI) system
\begin{align}
\begin{split} \label{OrigSys}
\bE\, \dot{\bx}(t) &= \bA\, \bx(t) + \bB\, \bu(t), \\
\by(t) &= \bC \bx(t),
\end{split}
\end{align}
where the input mapping is given by $\bu : \mathbb{R} \rightarrow \mathbb{R}^m$, the (generalized) state trajectory/variable is
$\bx: \mathbb{R} \rightarrow \mathbb{R}^n$, and the output mapping is $y : \mathbb{R} \rightarrow \mathbb{R}^p$. The system matrices are given by $\bE,\,\bA \in \mathbb{R}^{n \times n}$, and  $\bB\in \mathbb{R}^{n\times m}$, $\bC \in \mathbb{R}^{p\times n}$. Assume $\bE$ is nonsingular and that the matrix pencil $(\bA,\bE)$ is \emph{asymptotically stable}, meaning that the eigenvalues of $(\bA,\bE)$ (or equivalently of $\bE^{-1}\bA$) are located in the open left half-plane.
The transfer function of the LTI system \eqref{OrigSys},
defined as 
\begin{equation}\label{TF_def}
\bH(s)=\bC (s \bE -\bA)^{-1} \bB,    
\end{equation}
is a $p\times m$ matrix-valued rational function in $s$.

The two fundamental quantities in \BT are the reachability and observability Gramians. 
The reachability Gramian $\bP$ provides a measure of how easily a state can be accessed from the zero state.
\textcolor{black}{In the time domain, $\bP$ is 
	given by 
	\begin{align} \label{gram_P}
	\begin{split}
	\bP = \int_{0}^{\infty} e^{\bE^{-1}\bA t} \bE^{-1}\bB \bB^T\bE^{-T} e^{\bA^T\bE^{-T} t} dt.
	\end{split}
	\end{align}
}
Taking $\imunit^2 = -1$,
this can be represented equivalently in the frequency domain as
\begin{align} \label{gram_P_freq}
\begin{split}
\bP =
\frac{1}{2\pi}\int_{-\infty}^{\infty} (\imunit \zeta\bE -\bA)^{-1} \bB \bB^T (-\imunit \zeta\bE^T -\bA^T)^{-1} d \zeta.
\end{split}
\end{align}

In a complementary way, the observability Gramian describes how easily a state
can be observed, and may be represented as $\bE^T\bQ\bE$ where $\bQ$ is defined as
\begin{align} \label{gram_Q}
\begin{split}
\bQ = \frac{1}{2\pi}\int_{-\infty}^{\infty} (-\imunit \omega\bE^T -\bA^T)^{-1} \bC^T \bC (\imunit \omega\bE -\bA)^{-1} d \omega.
\end{split}
\end{align}
The Gramians $\bP$ and $\bQ$ satisfy the following generalized Lyapunov equations:
\begin{align}  \label{lyapforP}
\bA \bP\bE^T + \bE \bP \bA^T + \bB\bB^T = \bfz, \\
\bA^T \bQ \bE + \bE^T \bQ \bA + \bC^T \bC = \bfz. \label{lyapforQ}
\end{align}
Although $\bE^T\bQ\bE$ is the observability Gramian in this case,  the Lyapunov equation
\eqref{lyapforQ}
for $\bQ$ enjoys a fundamental structural similarity to 
\eqref{lyapforP} and so similar computational algorithms are applicable
(see, {e.g., \cite{penzl1998numerical,Stykel08,BenS13}}).  
The solutions, $\bP$ and $\bQ$, to \eqref{lyapforP} and \eqref{lyapforQ}, respectively, are symmetric  positive-semidefinite matrices, and one may compute square factors $\bL, \bU \in \mathbb{R}^{n \times n}$ (e.g., via a Cholesky factorization) so that
\begin{equation} \label{lyapFact}
\bP = \bU \bU^T \quad \mbox{and} \quad \ \bQ = \bL \bL^T.
\end{equation}
This provides the essential elements allowing us to define \emph{balanced truncation} (\BT), 
which we summarize in Algorithm \ref{origbt}. 
\begin{algorithm}[htp] 
	\caption{Balanced truncation (\BT) (square-root implementation)}  
	\label{origbt}                                        
	\algorithmicrequire~LTI system described by matrices $\bE,\,\bA \in \mathbb{R}^{n \times n}, \ \bB\in \mathbb{R}^{n \times m}$, and $\bC\in \mathbb{R}^{p \times n}$.
	
	\algorithmicensure~\BT reduced-system:   $\bA_r \in \mathbb{R}^{r \times r},  \bB_r\in \mathbb{R}^{r \times m}, \bC_r \in \mathbb{R}^{p \times r}$.
	
	\begin{algorithmic} [1]                                        
		\STATE \label{computeUL} Compute the  Lyapunov factors $\bU, \bL \in \mathbb{R}^{n\times n}$ from \eqref{lyapFact} and pick a 
		truncation index, $1\leq r\leq \min(\mathsf{rank}(\bU),\mathsf{rank}(\bL))$.
		\STATE  \label{svdstep} Compute the SVD of the matrix $ \bL^T \bE \bU$, partitioned as follows 
		\begin{equation}\label{SVD_BT}
		\bL^T \bE \bU = \left[ \begin{matrix}
		\bZ_1 & \bZ_{2}
		\end{matrix}  \right] \left[ \begin{matrix}
		\bS_1 & \\[1ex] & \bS_{2}
		\end{matrix}  \right] \left[ \begin{matrix}
		\bY_1^T \\[1ex] \bY_{2}^T
		\end{matrix}  \right],
		\end{equation}
		where $\bS_1 \in \mathbb{R}^{r \times r} \ \mbox{and} \ \bS_{2} \in \mathbb{R}^{(n-r) \times (n-r)}$.
		
		\STATE \label{WrVrstep}  
		Construct the model reduction bases 
		\begin{equation}
		\bW_r = \bL \bZ_1 	\bS_1^{-1/2} \ \mbox{and}\ \bV_r = \bU \bY_1 	\bS_1^{-1/2}.
		\end{equation}
		
		\STATE \label{project}  The reduced-order system matrices are given by
		\begin{equation} \label{btar}
		\hspace*{-6mm}{\small \begin{array}{cc}
			\bE_r = \bW_r^T \bE \bV_r= \bI_r   &   
			\bA_r = \bW_r^T \bA \bV_r=	\bS_1^{-1/2}\bZ_1^T\,(\bL^T\bA\bU)\,\bY_1\bS_1^{-1/2}, 
			\\[2mm]
			\bB_r = \bW_r^T \bB=\bS_1^{-1/2}\bZ_1^T\,(\bL^T\bB),
			& \bC_r = \bC \bV_r=(\bC\bU)\,\bY_1\bS_1^{-1/2}. 
			\end{array} }
		\end{equation}
		
	\end{algorithmic}
\end{algorithm}
The singular values of $\bL^T \bE \bU$ (diagonal entries of $\mathsf{diag}(\bS_1,\bS_2)$ in \eqref{SVD_BT}) are the \emph{Hankel singular values} of the associated dynamical system and these values are system invariants that are independent of realization. \BT truncates the states that correspond to small Hankel singular values in
$\bS_2$.  The $r$th-order reduced model resulting from \BT in Algorithm \ref{origbt} is asymptotically stable and its transfer function is
\begin{equation}\label{TFr_def}
\bH_r(s) = \bC_r(s\bE_r-\bA_r)^{-1}\bB_r,
\end{equation}
which satisfies
$
\| \bH - \bH_r \|_{\mathcal{H}_\infty} \leq 2\, \mathsf{trace}(\bS_2),
$
where $\| \cdot \|_{\mathcal{H}_\infty}$ denotes the ${\mathcal{H}_\infty}$-norm of a dynamical and where we assumed, for simplicity, that the Hankel singular values are distinct. For details, we refer the reader to \cite{ACA05}.

\section{A Data Driven Framework for Balancing} \label{sec:QuadBT}

The main innovation that we introduce centers on the observation that the key quantities $\bL^T\bE\bU$, 
$\bL^T\bA\bU$,  $\bL^T\bB$, and $\bC\bU$, appearing in Steps 1 and 3 of Algorithm \ref{origbt}, may be replaced by unitarily equivalent quantities that are well approximated directly from data. 
``Data'' in our setting correspond to sampling of the transfer function $\bH(s)$;
we assume access to the values of $\bH(\imunit \hat{\omega})$ at a finite number of frequencies $\hat{\omega}$.  
{(We consider time-domain sampling in \S \ref{sec:timeQBT}.)}
Our choice of frequency sampling, $\hat{\omega}$, will be associated with 
numerical quadratures used to approximate \eqref{gram_P_freq} and \eqref{gram_Q}.  Further details are provided below and in Appendix A, however we note here that
some quadrature rules also engage asymptotics of the integrand which in our setting may then require measuring the leading two Markov parameters of $\bH(s)$, i.e., the leading two coefficients in the expansion of $\bH(s)$ around $s=\infty$. 
For $\bH(s) = \bC(s\bE-\bA)^{-1}\bB$,  $\bM_0 = \bC \bE^{-1} \bB$ is the zeroth Markov parameter (the first coefficient) and $\bM_1 = \bC \bE^{-1} \bA \bE^{-1} \bB$ is the first Markov parameter (the second expansion coefficient). {These quantities can be measured nonintrusively, as explained in, e.g., \cite{Juang94} with practical applications ranging from Vector Network Analyzers \cite{keysight} to 3D Laser Vibrometers. \cite{Vibrometry14}.}

To simplify initial discussion, we consider first SISO systems, taking $\bB=\bb$ and $\bC=\bc^T$, for $\bb,\,\bc\in\mathbb{R}^n$ and the associated (scalar-valued) transfer function, $H(s)$\footnote{We distinguish matrix- and vector-valued quantities from scalar-valued quantities with boldface; hence $H(s)$ for SISO vs. $\bH(s)$ for the MIMO counterpart, and similarly for Markov parameters, $M_0$ and $M_1$ for SISO vs. $\bM_0$ and $\bM_1$ for the MIMO counterparts.}. Generalization to MIMO systems is straightforward and discussed in \S\ref{MIMOsys}. 

\subsection{Computing key quantities from data}  \label{quantitiesfromdata}

For large-scale dynamical systems, solving the Lyapunov equations \eqref{lyapforP} and \eqref{lyapforQ} for  the Gramians $\bP$ and $\bQ$  is computationally demanding. Sophisticated strategies have been developed for solving these equations; see, e.g., \cite{ACA05,BBF14,sabino2007solution,BenS13,simoncini2016computational,kurschner2016efficient} and  references therein. Using any of these techniques, one may find low-rank  approximations to $\bU$ and $\bL$, and then Algorithm \ref{origbt} uses instead these approximate low-rank factors.  These approximate balancing approaches remain intrusive in the sense that they still will depend on explicit system realizations and model projection, i.e., $\bU$ and $\bL$ are approximated first, followed by an explicit projection step.  Our approach will sidestep this through direct, non-intrusive, data-driven estimation of a quantity that is unitarily equivalent to $\bL^T \bE \bU$. 

We consider first a numerical quadrature rule that approximates the frequency integral defining $\bP$ in \eqref{gram_P_freq}, producing an approximate Gramian 
\begin{align} \label{quad_P}
\bP \approx \quadP = \sum_{j=1}^{\npf} \rho_j^2  (\imunit \zeta_j \bE -\bA)^{-1} \bb \bb^T (-\imunit \zeta_j \bE^T -\bA^T)^{-1}
+ \rho_\infty^2 \bE^{-1}\bb \bb^T \bE^{-T}
\end{align}
with $\rho_j^2$ and $\zeta_j$ denoting, respectively, quadrature weights and nodes. 
The last term involving $\rho_\infty$ corresponds to a ``node at infinity" and relates to the asymptotic decay of the integrand in \eqref{gram_P_freq} in the neighborhood of $\zeta = \infty$. 
The presence of this term in \eqref{quad_P} will depend on the choice of numerical quadrature; we investigate choices where this term is present and others where it is absent. In either case, we label the total number of quadrature points (including potentially a ``node at infinity") as $\np$.  If the $\rho_\infty$-term in \eqref{quad_P} is absent, then $\npf = \np$; otherwise, $\npf = \np-1$. 
Picking a distribution of quadrature nodes on the imaginary axis is natural in this context; we focus on two standard exemplars derived from the trapezoid rule and Clenshaw-Curtis quadrature and provide further details in Appendix A. Quadrature node distributions can be chosen on different integration contours as well (e.g.,\cite{weidemanTrefethen2007paraHyperContours}) and the approach that we propose is equally applicable in these modified circumstances as well. 

The \textsf{Balanced POD} approach of  \cite{WP02}  uses this type of approximation  for 
$\bP$, employing the trapezoid rule as an underlying numerical quadrature approximation.
A modified version of
\textsf{Balanced POD} based on a time-domain quadrature applied to \eqref{gram_P} was later proposed in  \cite{Ro05} with an application to fluid dynamics.  As is the case with other approximate 
balancing techniques, \textsf{Balanced POD} is \emph{intrusive} in the sense that access to an explicit state-space realization is required.

Approximate balancing techniques typically approximate the Lyapunov factors $\bU$ and $\bL$ from \eqref{lyapFact}, and to the extent that system response data may be utilized, its use seeks to reconstruct 
evolving states of the system and so, tacitly requires a well-defined state space realization.  This contrasts significantly with the \emph{nonintrusive} data-driven approach that we propose below.  While there are important common themes that we adopt from \textsf{Balanced POD}, we do not require any knowledge of a state space realization; we do not sample state trajectories even implicitly, and we have no need for direct approximation of the Lyapunov factors in \eqref{lyapFact}.

Evidently, we may decompose the quadrature-based Gramian approximation  as  $\quadP = \quadU \quadU^*$ with a square-root factor $\quadU \in \mathbb{C}^{n \times \np}$ defined as 
\begin{equation} \label{quad_U}
\quadU = \left[ \,
\rho_1  (\imunit \zeta_1 \bE -\bA)^{-1} \bb \quad 
\cdots \quad \rho_\npf  (\imunit \zeta_\npf \bE-\bA)^{-1} \bb 
\quad \rho_\infty \bE^{-1}\bb  \right].
\end{equation}
Note that both $\bP$ and its quadrature approximation, $\quadP$, are real-valued matrices, yet the explicit square-root factor, $\quadU$, is overtly complex and subsequent computation engages complex floating point arithmetic. A reasonable concern that may arise at this point is the potential loss of an underlying structural system symmetry, leaving us possibly with a complex-valued reduced model as an artifact of rounding errors. We  address this concern in \S\ref{sec:realness} and show how the underlying system symmetry associated with real dynamics may be explicitly preserved
and  implemented in real arithmetic.

By applying a  similar quadrature approximation to $\bQ$, we obtain
\begin{align} \label{quad_Q}
\begin{split}
\bQ \approx \quadQ = \sum_{k=1}^{\nqf} \phi_k^2  (-\imunit \omega_k \bE^T -\bA^T)^{-1} \bc \bc^T (\imunit \omega_k \bE -\bA)^{-1} + \phi_\infty^2 \bE^{-T}\bc \bc^T \bE^{-1},
\end{split}
\end{align}
where $\phi_k^2$ and $\omega_k$ denote, respectively, quadrature weights and nodes associated with approximating $\bQ$ from \eqref{gram_Q}.  The corresponding square-root factor is
$\quadQ = \quadL \quadL^*$ where
\begin{equation}  \label{quad_L}
\quadL^* = \left[ \begin{matrix}
\phi_1 \bc^T (\imunit \omega_1\bE -\bA)^{-1} \\
\vdots \\ \phi_\nqf \bc^T (\imunit \omega_\nqf \bE -\bA)^{-1}  \\[1mm] \phi_\infty \bc^T \bE^{-1}
\end{matrix}  \right] \in \mathbb{C}^{\nq \times n}.
\end{equation}

Let us suppose for the moment that the quadratures described in \eqref{quad_P} and  \eqref{quad_Q} \emph{exactly} 
recover the Gramians described in \eqref{gram_P_freq} and \eqref{gram_Q}, respectively, so that 
\begin{equation} \label{exactGramFact}
\bU\bU^T=\quadU\quadU^*\quad\mbox{and}\quad\bL\bL^T=\quadL\quadL^*.  
\end{equation}
This, in turn,  assures us of the existence of partial isometries, $\boldsymbol{\Psi}_p\in\mathbb{C}^{\np \times n}$ and $\boldsymbol{\Psi}_q\in\mathbb{C}^{\nq \times n}$ such that 
$$
\quadU^*=\boldsymbol{\Psi}_p\bU^T   \quad \mbox{and} \quad\quadL^*=\boldsymbol{\Psi}_q\bL^T.
$$
Thus, $\quadL^* \bE \quadU=\boldsymbol{\Psi}_q(\bL^T\bE\bU)\boldsymbol{\Psi}_p^*$ and so, 
$\quadL^* \bE \quadU$ is unitarily equivalent to $\bL^T\bE\bU$, modulo its kernel and cokernel. 
More precisely,
$\bL^T\bE\bU$ is isomorphic to a linear transformation induced by $\quadL^* \bE \quadU$ viewed as a mapping of the quotient space,  $\mathbb{C}^{\np}/\mathsf{Ker}(\quadU)$ onto $\mathsf{Ran}(\quadL^* \bE \quadU)$.
This is a significant observation since the singular values of $\bL^T\bE\bU$ are the \emph{Hankel singular values} of the system which are system invariants (i.e., independent of system realization) and play a fundamental role in \BT (see,  Step \ref{svdstep} of Algorithm \ref{origbt}).   
Although $\bL^T\bE\bU$ itself is evidently tied closely to knowledge of a state-space realization, 
the unitarily equivalent matrix, $\quadL^* \bE \quadU$ will provide essentially equivalent information while being directly derivable from data, as we will see.

\begin{proposition} \label{prop:quadL}
	Let $\quadU$ and $\quadL$ be as defined in 
	\eqref{quad_U} and \eqref{quad_L}. 
	Define the matrix $\quadLL =\quadL^* \bE \quadU \in \IC^{\nq \times \np}$. Then,
	for $1\leq k\leq\nq$ and $1\leq j\leq \np$,
	\begin{equation} \label{quadLoew}
	\quadLL_{k,j} = \begin{cases} - \phi_k \rho_j 
	\displaystyle \frac{H(\imunit\omega_k) - H(\imunit\zeta_j)}{\imunit\omega_k - \imunit \zeta_j}, \ \text{for} \ 1 \leq k \leq \nqf, \ 1 \leq j \leq \npf, \\[2mm]
	\phi_k \rho_\infty H(\imunit\omega_k), \ \text{for} \ 1 \leq k \leq \nqf = N_q, \  j = N_p = \npf +1, \\[2mm]
	\phi_\infty \rho_j H(\imunit\zeta_j), \ \text{for} \  k = N_q = \nqf +1, \ 1 \leq j \leq \npf = N_p, \\[2mm]
	\phi_\infty \rho_\infty M_0, \ \text{for} \  k = N_q = \nqf +1, \ j = N_p = \npf+1.
	\end{cases}
	\end{equation}
\end{proposition}
\begin{proof}
	Let $\bfe_k$ denote the $k$th canonical vector (of
	conforming length). From the definitions of 
	$\quadL$ and $\quadU$, and for indices $1 \leq k \leq \nqf$ and  $1 \leq j \leq \npf$, it follows that
	\begin{align*}
	\quadLL_{k,j} &= \bfe_k^T \quadLL \bfe_j =  
	(\bfe_k^T\quadL^*) \bE (\quadU\bfe_j)  = \phi_k \rho_j \bc^T (\imunit \omega_k \bE -\bA)^{-1}\,  \bE\, (\imunit \zeta_j \bE -\bA)^{-1} \bb \\ 
	& = \frac{\phi_k \rho_j}{\imunit\omega_k - \imunit \zeta_j} \bc^T(\imunit \omega_k \bE -\bA)^{-1} \left[ (\imunit\omega_k\bE-\bA) - (\imunit \zeta_j\bE-\bA) \right](\imunit \zeta_j \bE -\bA)^{-1}\bb \\ 
	& = -\frac{\phi_k \rho_j}{\imunit\omega_k - \imunit \zeta_j} \bc^T\left[ (\imunit\omega_k\bE-\bA)^{-1} - (\imunit \zeta_j\bE-\bA)^{-1} \right]\bb = - \phi_k \rho_j 
	\frac{H(\imunit\omega_k) - H(\imunit\zeta_j)}{\imunit\omega_k - \imunit \zeta_j}.
	\end{align*}
	Next, if $1 \leq k \leq \nqf = N_q$, and $j = N_p = \npf +1$ we can write that
	\begin{align*}
	\quadLL_{k,j} &= \bfe_k^T \quadLL \bfe_j =  
	(\bfe_k^T\quadL^*) \bE (\quadU\bfe_j)   = \phi_k  \bc^T (\imunit \omega_k \bE -\bA)^{-1}\,  \bE\, \rho_\infty \bE^{-1}\bb \\ 
	&= \phi_k \rho_\infty \bc^T (\imunit \omega_k \bE -\bA)^{-1} \bb = \phi_k \rho_\infty H(\imunit \omega_k).
	\end{align*}
	Similarly, we can derive $\quadLL_{k,j}$ for the case $k = N_q = \nqf +1$ and $1 \leq j \leq \npf = N_p$. Finally, for $ k = N_q = \nqf +1, \ j = N_p = \npf+1$, it follows that:
	\begin{align*}
	\quadLL_{k,j} =  
	(\bfe_k^T\quadL^*) \bE (\quadU\bfe_j)   = \phi_\infty \bc^T \bE^{-1}\,  \bE\, \rho_\infty \bE^{-1}\bb = \phi_\infty \rho_\infty \bc^T \bE^{-1} \bb = \phi_\infty \rho_\infty M_0,
	\end{align*}
	which completes the proof. \end{proof}

The observation that the quadrature-based approximate quantity $\quadLL=\quadL^*\bE \quadU$ may be obtained solely from transfer function samples is consistent with it being a
system invariant, producing
approximations to Hankel singular values that are also system invariants
and not dependent on any specific state-space realization.  Note that $\quadLL$ has at most rank-$n$, i.e., the McMillan degree of the underlying system. 

Of course, it will not be the case that the quadratures described in \eqref{quad_P} and  \eqref{quad_Q} are exact, so before we proceed we seek some assurance that the error induced by the quadrature approximation can be controlled.  We have the following result:
\begin{proposition} \label{prop:QuadError}
	Suppose the quadratures in \eqref{quad_P} and  \eqref{quad_Q} produce approximations $\quadP$ and 
	$\quadQ$ to $\bP$ and $\bQ$, respectively, that satisfy
	$\|\bQ-\quadQ\|_F\leq \frac{\delta}{1+\delta}\,\sigma_{\mathsf{min}}(\bQ)$ and 
	$\|\bP-\quadP\|_F\leq \frac{\delta}{1+\delta}\,\sigma_{\mathsf{min}}(\bP)$ for some $\delta\in(0,1)$ where
	$\sigma_{\mathsf{min}}(\ \cdot\ )$ denotes the smallest singular value. 
	Then there exist isometries  $\boldsymbol{\Psi}_p\in\mathbb{C}^{\np \times n}$ and $\boldsymbol{\Psi}_q\in\mathbb{C}^{\nq \times n}$ such that 
	\begin{equation} \label{quadBound}
	\|\quadL^* \bE \quadU-\boldsymbol{\Psi}_q(\bL^T\bE\bU)\boldsymbol{\Psi}_p^*\|_F
	\leq 2 \ \delta\ \|\bE\|_2\,\|\bL\|_2\,\|\bU\|_2.
	\end{equation}
	Let $\sigma_1\geq \sigma_2\geq \cdots\geq \sigma_n$ denote singular values of $\bL^T\bE\bU$ (i.e., Hankel singular values of  \eqref{OrigSys}), and let $\widetilde{\sigma}_1\geq \widetilde{\sigma}_2\geq \cdots\geq \widetilde{\sigma}_n$ denote the singular values of $\quadL^* \bE \quadU$.  Then,
	$$
	\left(\sum_{k=1}^n(\sigma_k-\widetilde{\sigma}_k)^2  \right)^{\frac12} \leq 2 \ \delta\ \|\bE\|_2\,\|\bL\|_2\,\|\bU\|_2.
	$$
\end{proposition}
The proof of this proposition may be found in Appendix B. 

This leads us to consider a \emph{quadrature-based} approximation to \BT 
that uses $\quadU$ and $\quadL$ in lieu of $\bU$ and $\bL$ in Step \ref{svdstep} of Algorithm \ref{origbt},
replacing the SVD of $\bL^T\bE\bU$ with
the SVD of the (nearly) unitarily equivalent matrix,   
$\quadLL=\quadL^* \bE \quadU$ that can be computed \emph{nonintrusively}, generated essentially from transfer function measurements.  

{Recall the singular values of $\bL^T\bE\bU$ (the {Hankel singular values}) are independent of system realization.}  Our nonintrusive, quadrature-based approach provides a new way of accessing the Hankel singular values without recourse to a system realization.  

\subsubsection{Numerical example}
\label{sec:exmphankel}
We use the heat model \cite{Niconet} (referred to as \textsf{[heat]} here),
describing thermal response of a thin rod - the system is SISO 
with dimension $n = 200$.

In Fig. \ref{fig:1}, we depict the true Hankel singular values together with the approximate ones obtained via the non-intrusive quadrature-based approach as described above. {We use two representative quadrature rules  described in Appendix A, labeled as 
	{\footnotesize \textsf{[ExpTrap]}} and {\footnotesize {\textsf{{[B/CC]}}}}.}
{As the figures illustrate, the quadrature-based estimates for the Hankel singular values approximate the true values  accurately.}

\begin{figure}[ht]
	\hspace{-5.5mm}
	\includegraphics[scale=0.205]{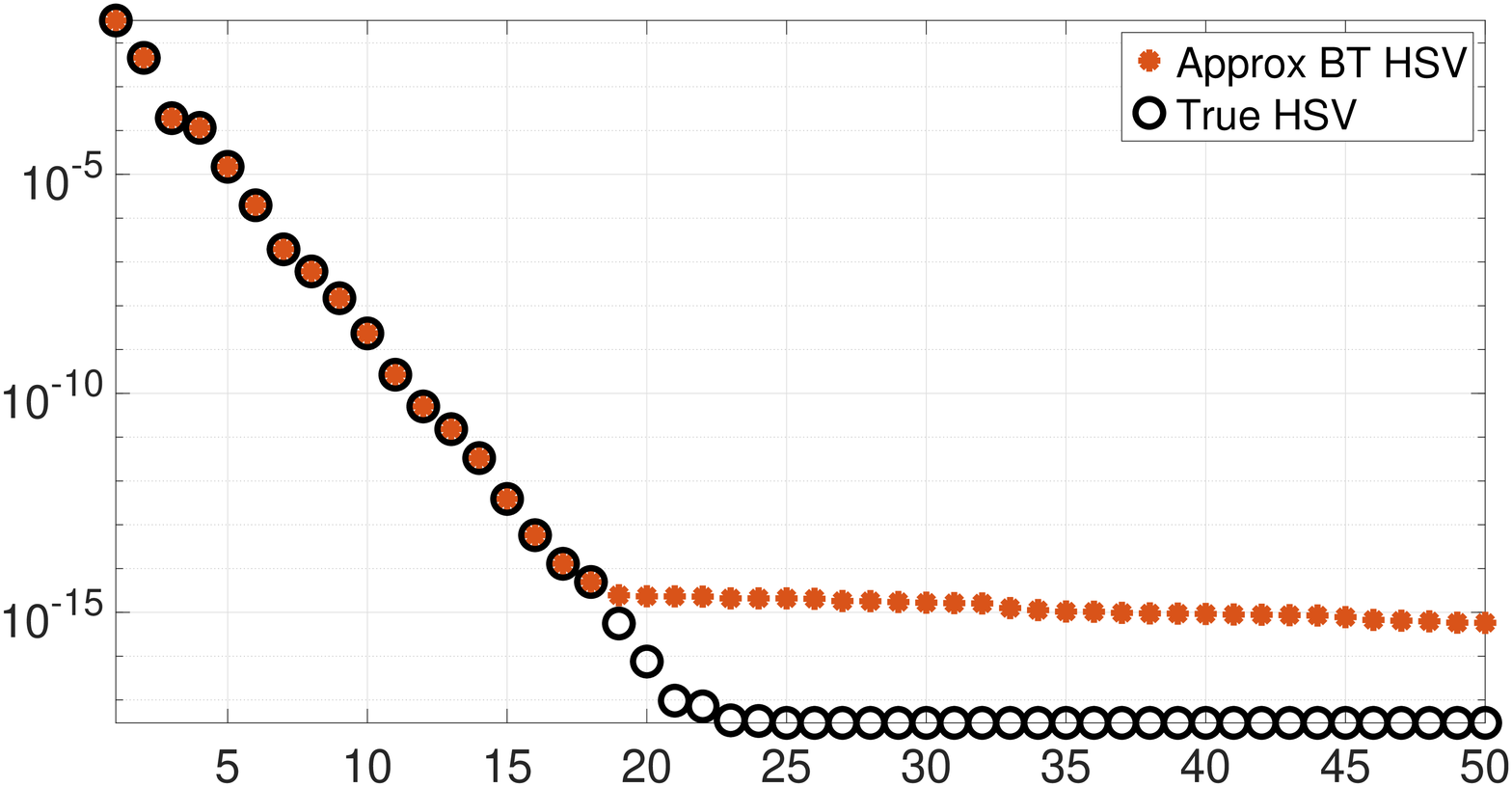} \hspace{-7.5mm}
	\includegraphics[scale=0.205]{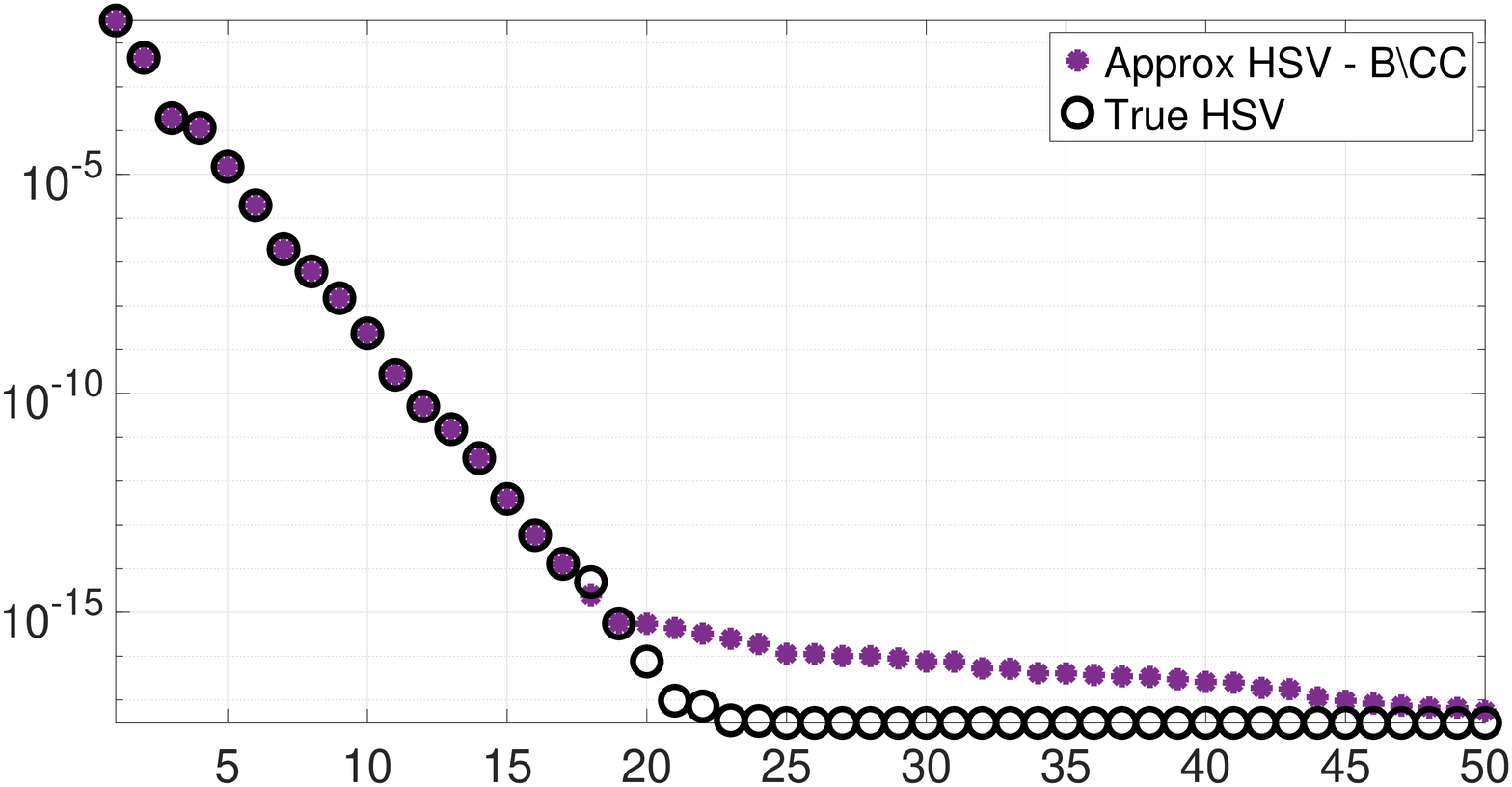}
	\vspace{-3mm}
	\caption{The computed (Hankel) singular values for the NICONET {\textrm\textsf{[heat]}} model using a trapezoid quadrature with 120 nodes  in the interval $[10^{-3},10^3] \imunit $  (left) and using 120 nodes with $L=4$ and $L = 3$ in {\textrm\textsf{[B/CC]}} for $\quadU$ and $\quadL$, respectively (right).}
	\label{fig:1}
	\vspace{-2mm}
\end{figure}

\subsection{Data-driven Balanced Truncation} \label{QuadBTruncation}
Although \BT yields reduced models in a manner that reflects system invariants and in that sense is independent of system realizations, \BT typically does require explicit access to a state-space representation and thus, in a practical sense, remains tied to system realizations. 
In particular, note that in Step \ref{WrVrstep} of Algorithm \ref{origbt}, model reduction bases $\bW_r$ and $\bV_r$ are constructed and then used to project state space quantities.  Through our (implicit) use of quadrature approximations for the reachability and observability Gramians, we are able to avoid any explicit reference to model reduction bases or to specific system realizations, and effect \BT in a coordinate system determined solely by input-output data.

Notice that in Step \ref{project} of Algorithm  \ref{origbt}, we have 
$
\bA_r = 	\bS_1^{-1/2}\bZ_1^T\,(\bL^T\bA\bU)\,\bY_1\bS_1^{-1/2},
$
and so the construction of $\bA_r$ makes reference to both $\bL^T \bA \bU$ and 
the SVD of $\bL^T\bE\bU$,  reflecting equivalent basis transformations on both $\bA$ and $\bE$.  Proposition \ref{prop:quadL} showed how we can replace
$\bL^T \bE \bU$ with the (nearly) unitarily equivalent matrix,   
$\quadLL=\quadL^* \bE \quadU$ (which could then be computed directly from data). 
So, we might expect that we could construct a realization equivalent to $\bA_r$ 
by replicating on $\bA$ an equivalent change of basis that had been visited upon $\bE$ by $\quadLL=\quadL^* \bE \quadU$. 
Pleasantly, we will find this to be the case by examining $\quadMM=\quadL^* \bA \quadU$.
Similar observations hold for $\bb_r$ and $\bc_r$.

\begin{proposition} \label{prop:quadM}
	Let $\quadU$ and $\quadL$ be as defined in 
	\eqref{quad_U} and \eqref{quad_L}. 
	Define the matrix $\quadMM =\quadL^* \bA \quadU \in \IC^{\nq \times \np}$. Then, the $(k,j)$ entry of $\quadMM$, 
	for $1\leq k\leq\nq$ and $1\leq j\leq \np$, is given by
	\begin{equation} \label{quadshiftLoew}
	\quadMM_{k,j} = \begin{cases} - \phi_k \rho_j \displaystyle \frac{\imunit\omega_k H(\imunit\omega_k) - \imunit \zeta_j H(\imunit\zeta_j)}{\imunit\omega_k - \imunit \zeta_j}, \ \text{for} \ 1 \leq k \leq \nqf, \ 1 \leq j \leq \npf, \\[2mm]
	\phi_k \rho_\infty \left( \imunit\omega_k H(\imunit\omega_k)-M_0\right), \ \text{for} \ 1 \leq k \leq \nqf = N_q, \  j = N_p = \npf +1, \\[2mm]
	\phi_\infty \rho_j \left( \imunit\zeta_j H(\imunit\zeta_j)-M_0 \right), \ \text{for} \  k = N_q = \nqf +1, \ 1 \leq j \leq \npf = N_p, \\[2mm]
	\phi_\infty \rho_\infty M_1, \ \text{for} \  k = N_q = \nqf +1, \ j = N_p = \npf+1.
	\end{cases}
	\end{equation}
	Likewise, defining $\quadLb = \quadL^*\bb
	\in \IC^{\nq \times 1}$ and $\quadcU^T = \bc^T \quadU\in \IC^{1 \times \np}$, we find
	\begin{eqnarray} 
	\quadLb_k & = & \begin{cases} \phi_k H(\imunit \omega_k), \ \text{for} \ 1 \leq k \leq \nqf,\\
	\phi_{\infty} M_0, \ \text{for} \ k = N_q = \nqf +1.
	\end{cases}
	\label{quadh} ~~\mbox{and} \\
	\quadcU_j & = &     \begin{cases} \rho_j H(\imunit \zeta_j), \ \text{for} \ 1 \leq j \leq \np, \\
	\rho_{\infty} M_0, \ \text{for} \  j= N_p = \npf +1.
	\end{cases}
	\label{quadg}
	\end{eqnarray}
\end{proposition}

\begin{proof}
	The proof is similar to that of Proposition \ref{prop:quadL}. Notice first that for any square matrix, $\bT$, and any $\varsigma,\,\mu \in\mathbb{C}$ that are not eigenvalues of $\bT$, one may verify
	$$
	\varsigma(\varsigma\bI-\bT)^{-1}-\mu(\mu\bI-\bT)^{-1}=-(\varsigma-\mu)(\varsigma\bI-\bT)^{-1}\bT(\mu\bI-\bT)^{-1}. 
	$$
	By setting $\bT=\bA\bE^{-1}$ and simplifying, one finds
	$$
	\varsigma(\varsigma\bE-\bA)^{-1}-\mu(\mu\bE-\bA)^{-1}=
	-(\varsigma-\mu)(\varsigma\bE-\bA)^{-1}\bA(\mu\bE-\bA)^{-1}. 
	$$
	Now, using the definitions of 
	$\quadL$ and $\quadU$,
	\begin{align*}
	\quadMM_{k,j} &= \bfe_k^T \quadMM \bfe_j =  
	(\bfe_k^T\quadL^*) \bA (\quadU\bfe_j) = \phi_k \rho_j \bc^T (\imunit \omega_k \bE -\bA)^{-1}  \bA (\imunit \zeta_j \bE -\bA)^{-1} \bb \\ 
	& = -\frac{\phi_k \rho_j}{\imunit\omega_k - \imunit \zeta_j} \bc^T\left[ \imunit\omega_k (\imunit\omega_k\bE-\bA)^{-1} - \imunit \zeta_j(\imunit\zeta_j\bE-\bA)^{-1} \right]\bb \\
	&= - \phi_k \rho_j \frac{\imunit\omega_k H(\imunit\omega_k) - \imunit \zeta_j H(\imunit\zeta_j)}{\imunit\omega_k - \imunit \zeta_j}.
	\end{align*}
	Next, if $1 \leq k \leq \nqf = N_q$, and $j = N_p = \npf +1$ we can write that
	\begin{align*}
	\quadMM_{k,j} &= \bfe_k^T \quadMM \bfe_j =  
	(\bfe_k^T\quadL^*) \bA (\quadU\bfe_j)  = \phi_k  \bc^T (\imunit \omega_k \bE -\bA)^{-1}  \bA \rho_\infty \bE^{-1} \bb \\ 
	&= \phi_k \rho_\infty  \bc^T\left[  \imunit\omega_k (i\omega_k\bE-\bA)^{-1} - \bE^{-1} \right]\bb =  \phi_k \rho_\infty \left( \imunit\omega_k H(\imunit\omega_k) - M_0\right).
	\end{align*}
	Similarly, we can find the exact derivation of $\quadLL_{k,j}$ in \eqref{quadshiftLoew} for the case $k = N_q = \nqf +1, \ 1 \leq j \leq \npf = N_p$. Next, for $ k = N_q = \nqf +1, \ j = N_p = \npf+1$, it follows that:
	\begin{align*}
	\quadMM_{k,j} &= \bfe_k^T \quadMM \bfe_j =  
	(\bfe_k^T\quadL^*) \bA (\quadU\bfe_j)   = \phi_\infty \bc^T \bE^{-1}\,  \bA\, \rho_\infty \bE^{-1}\bb \\ &= \phi_\infty \rho_\infty \bc^T \bE^{-1} \bA \bE^{-1} \bb = \phi_\infty \rho_\infty M_1.
	\end{align*}
	Finally, one can write
	\begin{align*}
	\quadLb_k = \bfe_k^T \quadLb =(\bfe_k^T\quadL) \bb = \begin{cases}   \phi_k\bc^T (\imunit \omega_k \bE-\bA)^{-1} \bb =  \phi_k H(\imunit\omega_k), \ \text{for} \ 1 \leq k \leq \nqf, \\
	\phi_\infty \bc^T \bE^{-1} \bb =  \phi_\infty M_0,  \ \text{for} \ k = N_q = \nqf +1.
	\end{cases}
	\end{align*}
	and
	\begin{align*}
	\quadcU_j = \quadcU^T \bfe_j = \bc^T 
	(\quadU \bfe_j) = \begin{cases}
	\rho_j \bc^T   (\imunit \zeta_j\bE -\bA)^{-1} \bb =  \rho_j H(\imunit\zeta_j), \ \text{for} \ 1 \leq j \leq \np, \\
	\rho_\infty \bc^T  \bE^{-1} \bb =  \rho_\infty M_0,  \ \text{for} \  j= N_p = \npf +1,
	\end{cases}
	\end{align*}
	completing the proof.
\end{proof}

We have been able to replace  
$\bL^T \bE \bU$, $\bL^T \bA \bU$, $\bL^T \bb$, and $\bc^T \bU$  in  Algorithm \ref{origbt}
with equivalent quantities that are derivable directly from data (i.e., from transfer function samples).
We assemble this together in Algorithm \ref{quadbt}, yielding a computationally feasible strategy for
quadrature-based balanced truncation (\QBT), requiring only transfer function sampling (i.e., no 
access to internal properties), and fully capable of recovering models equivalent to \BT reduced models to any accuracy desired.

\begin{algorithm}[htp] 
	\caption{Quadrature-based (data-driven) balanced truncation (\QBT)}  
	\label{quadbt}                                     
	\algorithmicrequire~LTI system described through a transfer function evaluation map, $H(s)$; \\
	\hspace*{5mm} quadrature nodes, $\zeta_j$, and weights, $\rho_j$, for $j=1,2,\ldots,\np$; \\
	\hspace*{5mm} quadrature nodes, $\omega_k$, and weights, $\phi_k$, for $k=1,2,\ldots,\nq$, \\
	\hspace*{5mm} and a truncation index, $1\leq r\leq \min(\np,\nq)$.
	
	\algorithmicensure~{\QBT} reduced-system:  $\quadAr\in\mathbb{R}^{r \times r}, \ \quadbr, \quadcr \in \mathbb{R}^r$.
	
	\begin{algorithmic} [1]                                        
		\STATE \label{sample} Sample transfer function values, $\{H(\imunit \zeta_j)\}_{j=1}^{\np}$  and $\{H(\imunit \omega_k)\}_{k=1}^{\nq}$.  Using the samples and quadrature weights, $\{\rho_j\}$ and $\{\phi_k\}$, 
		construct $\boldsymbol{\quadLL} \in \IC^{\nq \times \np}$, $\boldsymbol{\quadMM} \in \IC^{\nq \times \np}$, $\quadLb$, and $\quadcU$  as in \eqref{quadLoew}, \eqref{quadshiftLoew}, \eqref{quadh}, and \eqref{quadg}, respectively, 
		\STATE Compute the SVD of $\quadLL$:
		\begin{equation}\label{SVD_BT2}
		\quadLL  = \left[ \begin{matrix}
		\quadZ_1 & \quadZ_{2}
		\end{matrix}  \right] \left[ \begin{matrix}
		\quadS_1 & \\ & \quadS_{2}
		\end{matrix}  \right] \left[ \begin{matrix}
		\quadY_1^* \\ \quadY_{2}^*
		\end{matrix}  \right],
		\end{equation}
		where $\quadS_1 \in \mathbb{R}^{r \times r}$ and $\quadS_{2} \in \mathbb{R}^{(\nq-r) \times (\np-r)}$.
		
		\STATE Construct the reduced order matrices:
		\begin{equation} \label{quadbtar}
		\begin{array}{cc} \quadEr= \quadS_1^{-1/2} \quadZ_1^* \ \quadLL \ \quadY_1   \quadS_1^{-1/2} = \bI_r, 
		\quad & \quad 
		\quadAr = \quadS_1^{-1/2} \quadZ_1^* \ \quadMM \ \quadY_1   \quadS_1^{-1/2}, \\[2mm]
		\quadbr =  \quadS_1^{-1/2} \quadZ_1^* \quadLb,\qquad \mbox{and} &\quad 
		\quadcr =  \quadcU^T \quadY_1   \quadS_1^{-1/2}.
		\end{array}
		\end{equation}
	\end{algorithmic}
\end{algorithm}
Note the contrast in the construction of $\bA_r$ in \eqref{btar} with that of $\quadAr$ in \eqref{quadbtar}. 
The former computes the inner term $\bL^T\bA\bU$ and requires access to internal quantities characterizing the dynamics; the latter approximates this term through $\quadMM$ and may be 
constructed directly from transfer function samples. Indeed, the quantities $\quadS_1^{-1/2} \quadZ_1^*$ and $\quadY_1 \quadS_1^{-1/2}$ are available from the SVD of the data matrix $\quadLL$ in \eqref{quadLoew}.

\begin{remark}
	In the analysis above, we have assumed that the quadrature nodes $\{\zeta_j\}$ for $\quadP$  
	and $\{\omega_k\}$ for $\quadQ$ are distinct from each other. 
	However, it may be more convenient to choose the same quadrature nodes (and weights) for both $\quadP$ and $\quadQ$. This leads to only minor changes
	to  Algorithm \ref{quadbt}.
	The vectors $\quadLb$ in \eqref{quadh} and $\quadcU$ in \eqref{quadg} stay unchanged. 
	The principal change occurs in the diagonal entries of $\boldsymbol{\quadLL}$ and
	$\boldsymbol{\quadMM}$: if $\omega_k = \zeta_k$, then the $(k,k)$ entry of  $\boldsymbol{\quadLL}$ in \eqref{quadLoew} becomes
	$$\boldsymbol{\quadLL}_{k,k} = 
	-\phi_k \rho_k H'(\imunit \omega_k),
	$$  
	where  $H'$ denotes the derivative of $H$. The $(k,k)$ entry of  $\boldsymbol{\quadMM}$ in
	\eqref{quadshiftLoew} is replaced by
	$$\boldsymbol{\quadMM}_{k,k} = 
	-\phi_k \rho_k \left(\omega_k H'(\imunit \omega_k)
	+ H(\imunit \omega_k) \right).
	$$  
	The remaining entries in $\boldsymbol{\quadLL}$ and 
	$\boldsymbol{\quadMM}$ are unchanged. 
\end{remark}

\subsubsection{Numerical examples - the SISO case}
\label{sec:firstcomp}
In this section we will test the approximation capability of our newly proposed method \QBT and compare it with that of classical \BT. We employ the following abbreviations: 
\begin{enumerate}[leftmargin=12pt,topsep=1pt,itemsep=-0.5ex,partopsep=0ex,parsep=1ex]
	\item \textsf{[BT-classic]}: The classical \BT approach  (Algorithm \ref{origbt}).
	\item \textsf{[ExpTrap-N]}: \QBT 
	(Algorithm \ref{quadbt})
	via the exponential trapezoid rule with $N$ points (without derivative samples); {more details may be found in Appendix A.} 
	\item \textsf{[B/CC-N]}: \QBT (Algorithm \ref{quadbt}) via Boyd/ Clenshaw-Curtis rule with $N$ points (without derivative samples); {more details may be found in Appendix A.} 
\end{enumerate}

We first consider the [\textsf{heat}] model from
\S\ref{sec:exmphankel}.  We use first $60$ and then $120$  logarithmically-spaced points in the interval $[10^{-3}, 10^3] \imunit$  as nodes for  [\textsf{ExpTrap}] (the left and right nodes are different). In [\textsf{B/CC}], choose  $L=4$ and $L = 3$  for $\quadU$ and $\quadL$, respectively, using again first 60 and then 120 nodes. We  vary the reduction order from $r=2$ to $r=14$ in increments of two and collect both the $\cH_\infty$  and  the $\cH_2$ norms of the error systems $\hot{\bH(s)-\bH_r(s)}$ corresponding to \textsf{[BT-classic]} and both quadrature schemes. The results are presented in Fig. \ref{fig:3}, illustrating that both
\QBT models (with \textsf{[ExpTrap-N]} and \textsf{[B/CC-N]}) nearly replicate the performance of \textsf{[BT-classic]}.

\begin{figure}[ht]
	\hspace{-7mm}
	\includegraphics[scale=0.205]{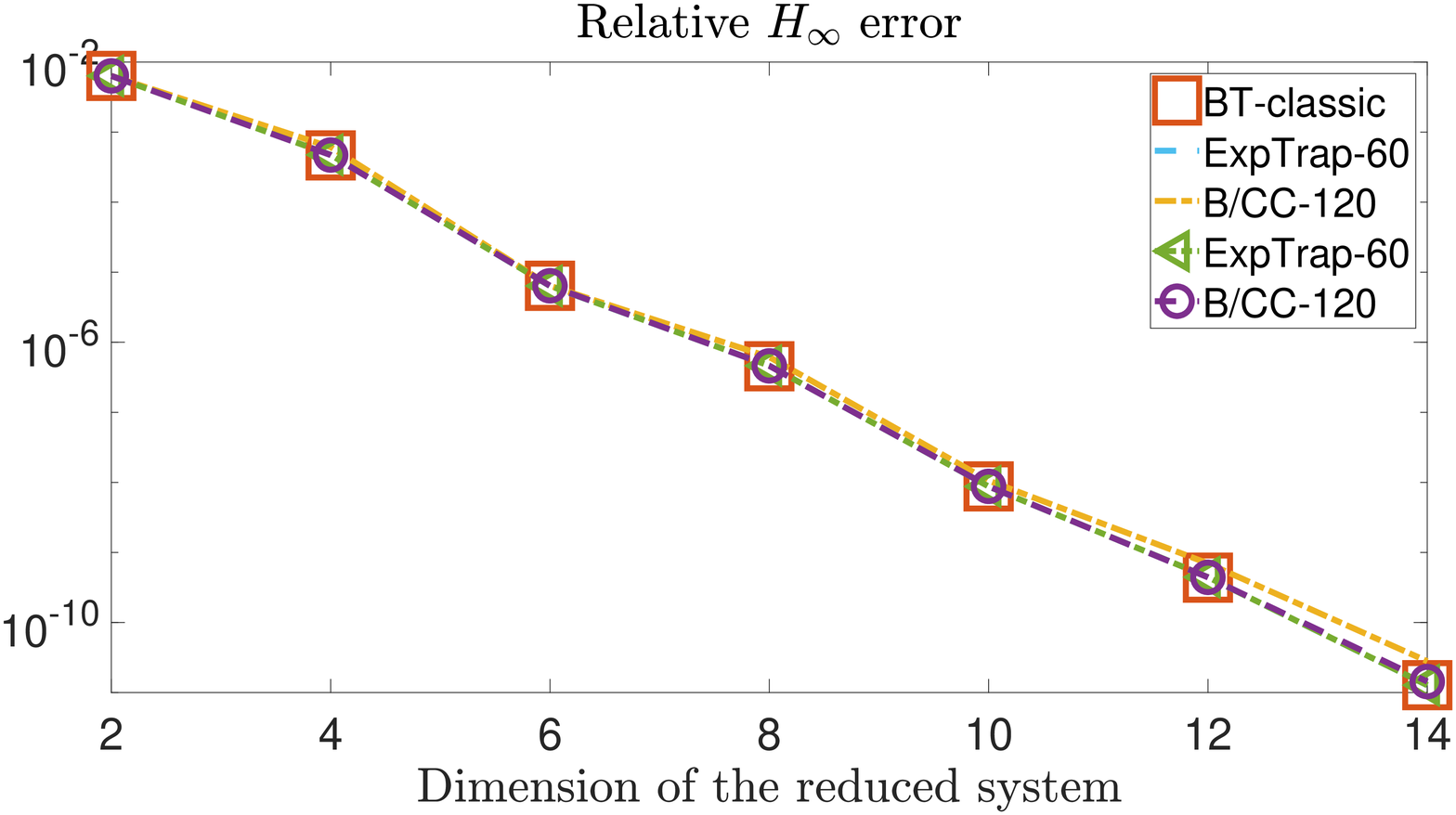}
	\hspace{-6mm}
	\includegraphics[scale=0.205]{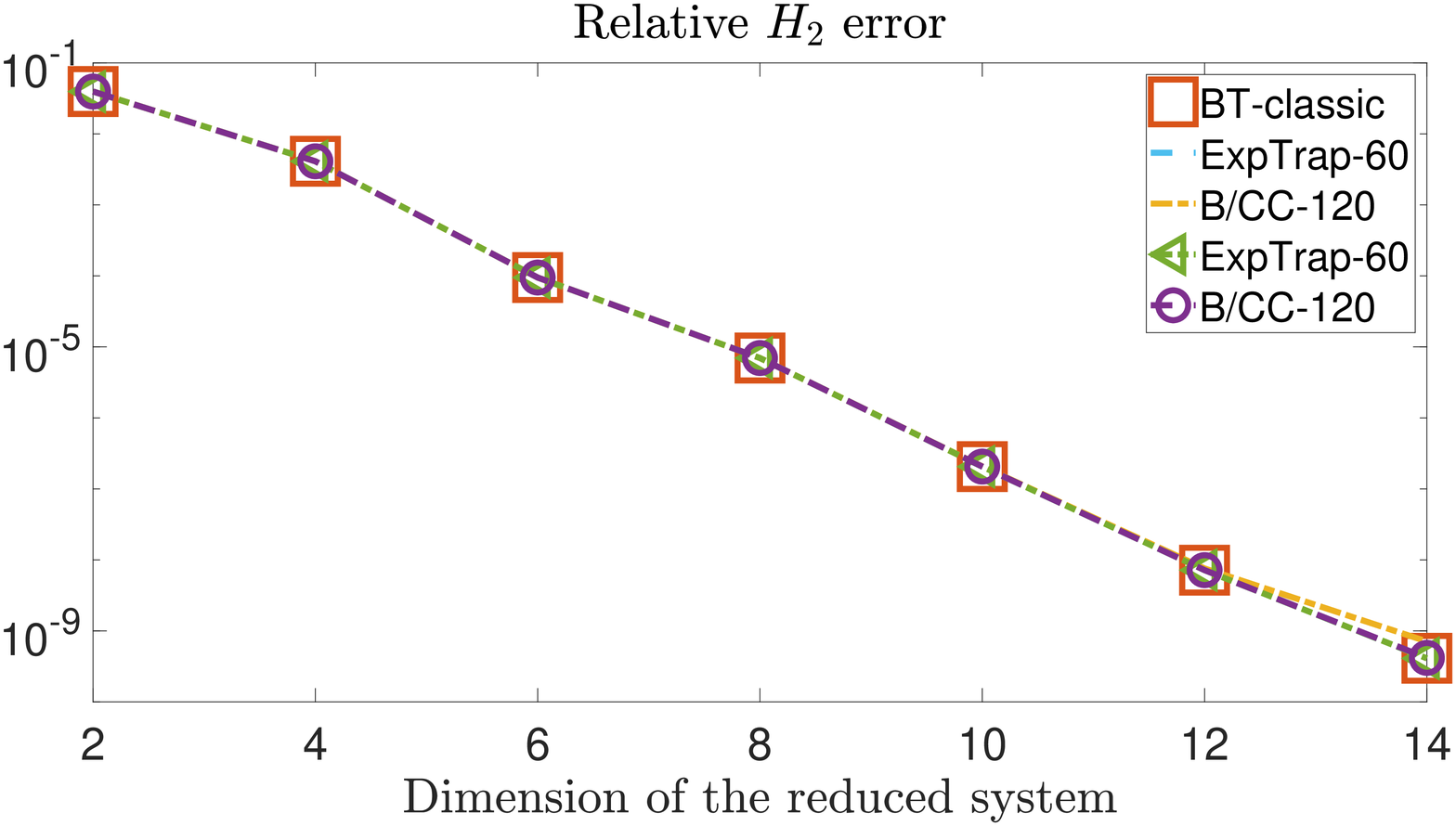}
	\vspace{-3mm}
	\caption{The relative $\cH_\infty$ (left) and $\cH_2$ (right) error approximation for different quadratures and nodes applied to the Niconet {\textrm [\textsf{heat}]} model.}
	\label{fig:2}
	\vspace{-2mm}
\end{figure}

As a second example, we use the International Space Station benchmark \cite{Niconet} (referred to as  \textsf{[iss1r]}) modeling the 1r component of the International Space Station --- the system dimension is $n = 270$; it has three inputs and three outputs.
To enforce a SISO system, we restrict our attention to the first input and first output. We choose both $200$ and $400$  logarithmically-spaced points in the interval $[10^{-1}, 10^2] \imunit$  as nodes for the [\textsf{ExpTrap}] scheme (the left and right nodes are different). In  [\textsf{B/CC}], we choose
$L=10$ and $L = 9$  for $\quadU$ and $\quadL$, respectively
and use both $200$ and $400$ nodes. We  then vary the reduction order in the interval $[2,24]$ (in increments of two) and collect both the $\cH_\infty$  and  the $\cH_2$ norms of the error systems corresponding to all three reduced models. The results are presented in Fig. \ref{fig:3}. 
{As in the [\textsf{heat}] case, these figures show that transfer-function based \QBT very closely mimics (in many cases exactly replicates) the $\cH_\infty$ and 
	$\cH_2$
	performance of the projection-based \BT. In this example, [\textsf{B/CC}] quadrature outperforms   [\textsf{ExpTrap}].
}

\begin{figure}[h]
	\hspace{-6mm}
	\includegraphics[scale=0.205]{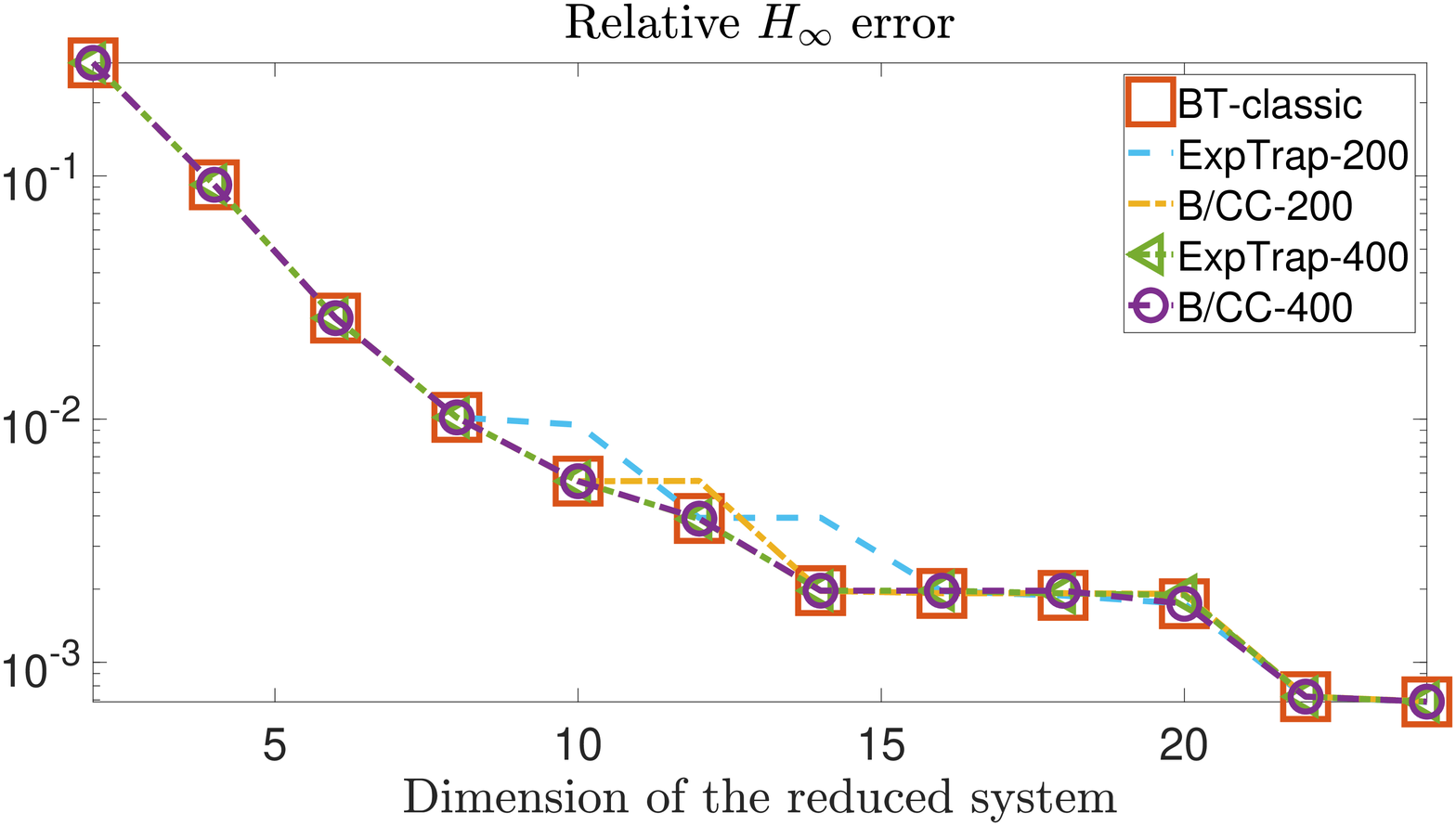}
	\hspace{-7mm}
	\includegraphics[scale=0.205]{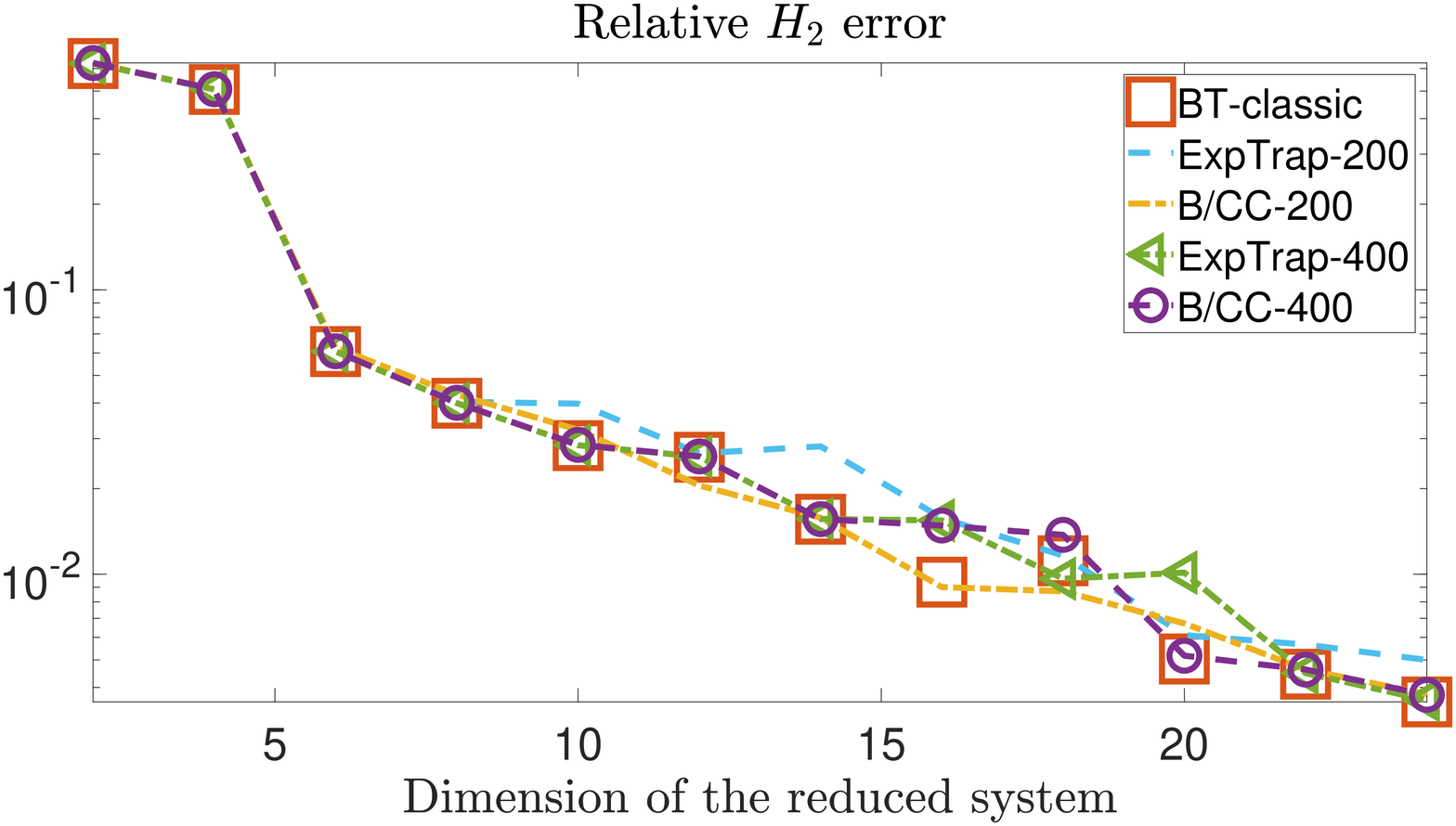}
	\vspace{-3mm}
	\caption{The relative $\cH_\infty$ (left) and $\cH_2$ (right) error approximation for different quadratures and nodes applied to the Niconet {\textrm [\textsf{iss1r}]} model.}
	\label{fig:3}
	\vspace{-4mm}
\end{figure}

For our last experiment, we purposefully reduce the accuracy of the underlying quadrature by using only a third as many nodes as in the previous experiment, while keeping the same nodal configuration as before (i.e., 20 and 40 nodes, respectively, that are logarithmically-spaced in $[10^{-3}, 10^3] \imunit$ for  [\textsf{ExpTrap}] and the same $L$-parameter choices  for [\textsf{B/CC}]). This leads to a slightly less accurate approximation of Hankel singular values, as can be observed in the left part of Fig. \ref{fig:2b}. Likewise, the corresponding reduced models are slightly less accurate than those computed in the first experiment, see the right part of Fig. \ref{fig:2b}. This situation is somewhat analogous to what is observed with approximate balanced truncation methods that approach the solution of the Lyapunov equations \eqref{lyapforP}-\eqref{lyapforQ} through explicit low rank approximation of the related square root factors, \eqref{lyapFact}.   
Such approaches, exemplified by low-rank ADI methods (see e.g., \cite{BenS13,DKS14,penzl1999cyclic,li2002low}, 
produce errors in the low-rank factors that jointly contribute to the error in $\mathbf{L}^T \mathbf{E} \mathbf{U}$ in the first step of BT, in a way that is similar to what we see when coarsening our quadrature rules.  Although the behavior of these projection-based low-rank ADI approaches is still not fully understood, there is overwhelming numerical evidence that projection-based approximate BT produced by such  methods works very well in practice. 
We anticipate that this favorable situation will be seen in our current data-driven BT formulation, as well. 

\begin{figure}[ht]
	\hspace{-6mm}
	\includegraphics[scale=0.205]{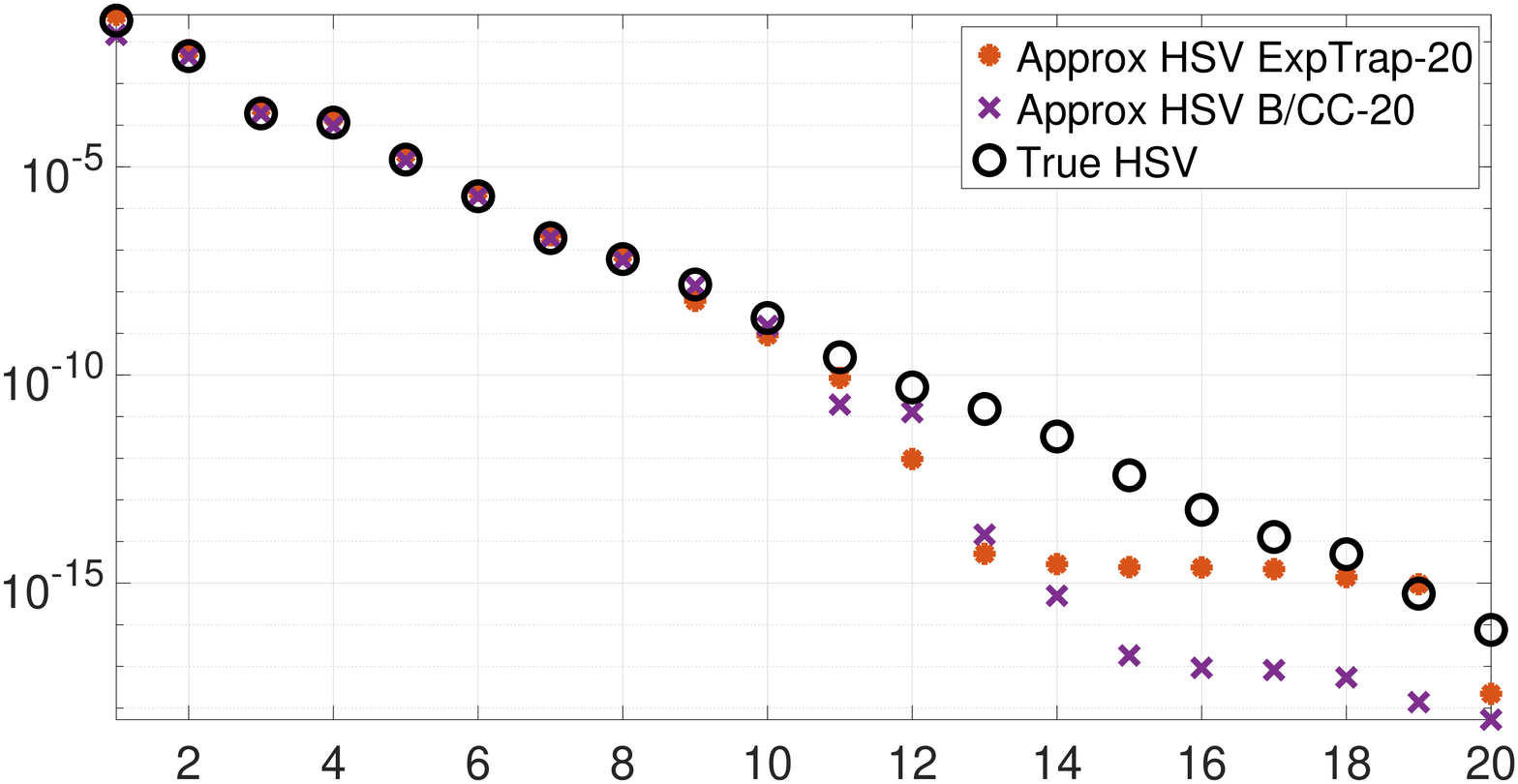}	
	\hspace{-6mm}
	\includegraphics[scale=0.205]{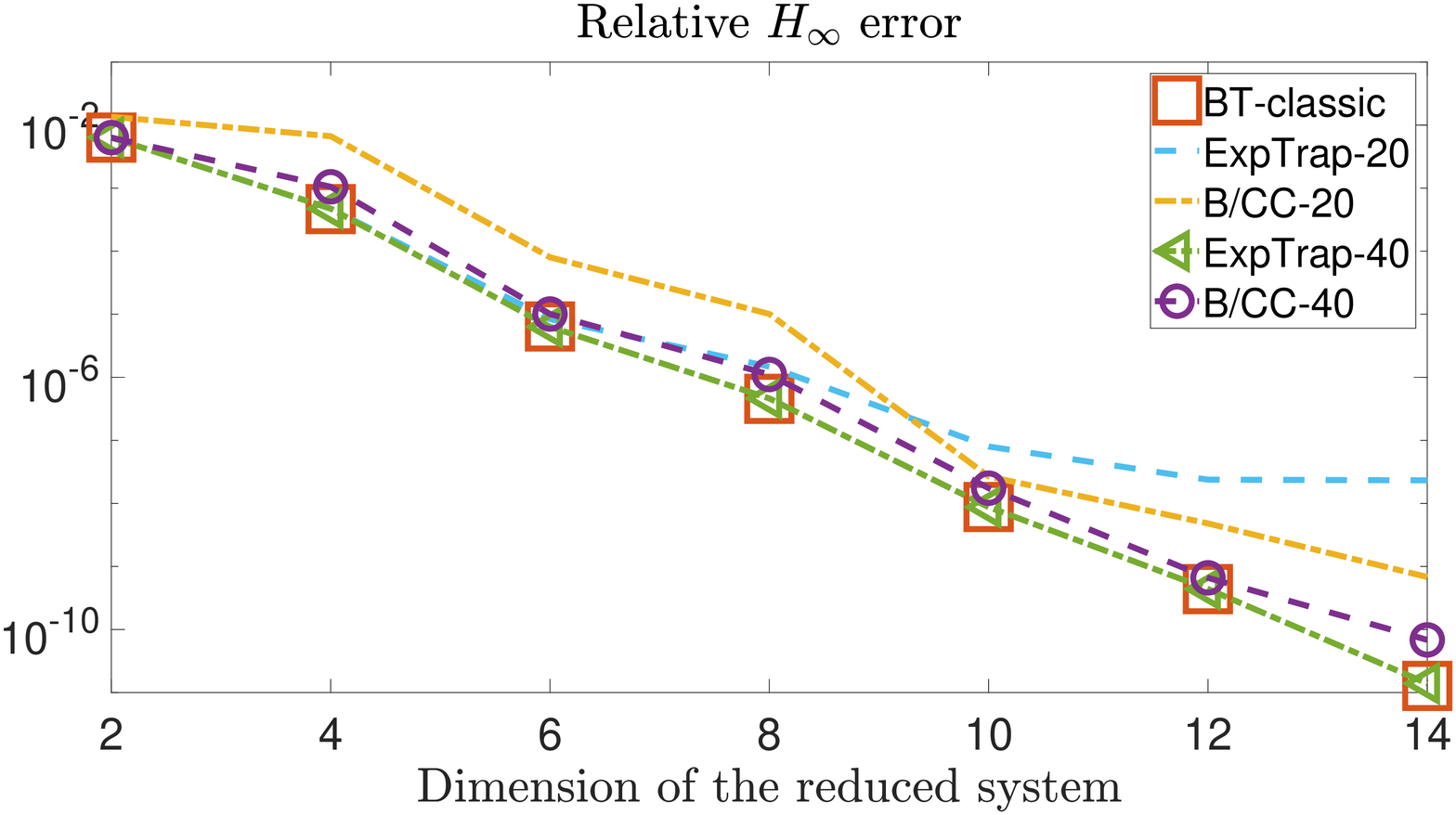}
	\vspace{-3mm}
	\caption{The computed (Hankel) singular values for the NICONET {\textrm\textsf{[heat]}} model using using 20 nodes with $L=4$ and $L = 3$ in {\textrm\textsf{[B/CC]}} for $\quadU$ and $\quadL$, respectively and 20 nodes for \textsf{[ExpTrap]} (left) and the relative $\cH_\infty$ (right) error approximation for different quadratures and nodes applied to the Niconet {\textrm [\textsf{heat}]} model.}
	\label{fig:2b}
	\vspace{-2mm}
\end{figure}

\subsection{Connections to approximate \BT methods}
\label{sec:othermethods}

Due to its compelling theoretical properties, an immense amount of literature has developed around effective methods for approximating \BT for large-scale dynamical systems, often focusing on efficient handling of the main computational bottleneck in \BT, namely, the computation of Gramians, which requires the solution of  two large-scale Lyapunov equations \eqref{lyapforP}-\eqref{lyapforQ}.  We list a small selection of these works here \cite{penzl1999cyclic,li2002low,Stykel08,BenS13,simoncini2016computational,kurschner2016efficient}; further references and detail can be found within these resources.

A common approach for approximate \BT
involves constructing  low-rank approximate factors 
to $\bP$ and $\bQ$,
i.e., $\bP \approx \lrU\lrU^T$ and 
$\bQ \approx \lrL\lrL^T$, with $\mathsf{rank}(\lrU)$ and $\mathsf{rank}(\lrL)$ much smaller than $n$.
Then we use $\lrU$ and $\lrL$ in Algorithm \ref{origbt} in place of the true square-root factors
$\bU$ and $\bL$.  
Unfortunately, regardless of how $\lrU$ and $\lrL$ are computed, 
this construction requires access to internal dynamics and state-space quantities.
The reduced model is then obtained by an explicit projection as in \eqref{btar}. 
This is fundamentally different from what we develop in this work; we do not compute $\lrU$ or $\lrL$ and we do not require access to state-space quantities. We only assume access to transfer function samples. Despite these major contrasts, there are striking similarities that motivate the underlying approximation theory.

The quadrature approximation \eqref{quad_P} to $\bP$  has been a major force behind the frequency-domain version of \textsf{Balanced POD}~\cite{WP02}. Indeed, the \textsf{Balanced POD} approach of \cite{WP02} engages a numerical quadrature for the Gramians, $\bP$ and $\bQ$, using unit quadrature weights;
$\rho_j = 1$ for $j=1,\ldots,\npf$ \eqref{quad_P}  and $\phi_k = 1$ for $j=1,\ldots,\nqf$ in \eqref{quad_Q}, and omitting the node at infinity in both cases.
The approach of \cite{WP02} was later referred to as Poor Man’s Truncated Balanced Reduction
in~\cite{phillips2004poor}; however, the implementation there is restricted to symmetric systems.
Significantly, \textsf{Balanced POD} is not data-driven in the sense that we explore here, but rather it is both \emph{projection-based} and \emph{intrusive},  requiring explicit access to a state-space representation. 
Nonetheless, it is clear that one may immediately arrive at a transfer-function based \emph{data-driven} formulation of  \textsf{Balanced POD} by using the \QBT framework developed here.

Along similar lines, Gauss-Kronrod quadrature was used to approximate Hankel singular values in ~\cite{benner2010balanced}, but as with earlier methods, explicit state-space representations 
are still required. Similar quadrature strategies were later used in~\cite{breiten2016structure}
to extend \BT to structure-preserving model reduction for integro-differential equations.  In all these methods, frequency-domain quadrature to approximate Gramians can be viewed as a common thread shared with our framework. Significantly, we avoid explicit state-space projections by directly working with  transfer function evaluations. 

Evidently, numerical quadrature can also be applied to the time-domain representation of $\bP$ as in~\eqref{gram_P} and this has also been considered as a strategy for approximate \BT, see, e.g.,  \cite{opmeer2011model,moore1981principal,BF19,BF20,singler2009proper,himpe2018emgr}. 
We explore these connections in \S \ref{sec:timeQBT} where a time-domain formulation of this data-driven \BT framework is considered.

ADI-type methods  (see, e.g.,\cite{smith1968matrix,penzl1999cyclic,li2002low,BKS14,DKS14,BenS13,simoncini2016computational}) are among the most  commonly used tools for producing  approximate low-rank Gramian factors $\lrU$ and $\lrL$. These methods are iterative in character and depend on  shift parameters that determine the speed of convergence. Even though the resulting low-rank factor $\lrU$ will have a different structure than $\quadU$ in \eqref{quad_U}, one may view either result as a low-rank factor for frequency-domain \textsf{Balanced POD}. Both factors span the same (Krylov) subspaces provided that the shift parameters are chosen to be the negative of the quadrature nodes \cite{BBF14}. Significantly, low-rank ADI solutions have also been  connected to time-domain quadrature methods \cite{BF19,BF20}, highlighting a  powerful approximation-theoretic core unifying the approaches.
Our approach is distinct from other methods in that it never constructs approximate solutions to $\bP$ or $\bQ$, nor attempts to approximate low-rank factors. Our \QBT reduced model is constructed directly from data.

\subsection{Comparison with interpolatory Loewner model reduction} 
\label{ClassicLoewner}

For the sampling points 
$\{\lp_k\}_{k=1}^\nq$ and $\{\rp_j\}_{j=1}^\np$, 
consider $\boldsymbol{\origLL} \in \IC^{\nq \times \np}$, $\boldsymbol{\origMM} \in \IC^{\nq \times \np}$, defined element-wise as: 
\begin{align} \label{origLoewmat}
\origLL_{k,j}& =  -  \frac{H( \lp_k) - H( \rp_j)}{ \lp_k -  \rp_j} \quad \mbox{and} \quad \origMM_{k,j} = -  \frac{ \lp_k H( \lp_k) -  \rp_j H( \rp_j)}{ \lp_k - \rp_j}, \\
& ~~\mbox{for}~~k=1,2,\ldots,\nq~~\mbox{and}~~j=1,2,\ldots,\np. \nonumber
\end{align}
The matrices $\origLL$  and $\origMM$ in \eqref{origLoewmat} are called, respectively, the Loewner and shifted-Loewner matrices.  Notice that $\origLL$  and $\origMM$ bear evident similarities to the key matrices
$\quadLL$ in \eqref{quadLoew}  and $\quadMM$ in \eqref{quadshiftLoew}   that are central to \QBT.
Indeed, assuming that the node at infinity is absent from the quadrature rules in \eqref{quad_P} and \eqref{quad_Q}, and that the quadrature nodes in \QBT determine {$\mu_k=\imunit\omega_k$ for $k=1,2,\ldots,\nq$ and $\lambda_j=\imunit\zeta_j$} 
for $j=1,2,\ldots,\np$ in \eqref{origLoewmat},  the matrices  $\quadLL$ in \eqref{quadLoew} and $\quadMM$ in \eqref{quadshiftLoew} used in \QBT will be diagonal scalings of $\origLL$ and 
$\origMM$:
{\small
	$$
	\quadLL = \mathsf{diag}(\phi_1,\ldots,\phi_\nq) \,\origLL \,
	\mathsf{diag}(\rho_1,\ldots,\rho_\np)
	\mbox{ and } 
	\quadMM = \mathsf{diag}(\phi_1,\ldots,\phi_\nq) \,\origMM \,
	\mathsf{diag}(\rho_1,\ldots,\rho_\np).
	$$
}
The matrices $\origLL$  and $\origMM$ of \eqref{origLoewmat} were originally conceived to play a role in quite a different model reduction scheme which we refer to as the \emph{interpolatory Loewner model reduction framework}. 
This model reduction framework was introduced in~\cite{mayo2007fsg} as a data-driven interpolation-based system identification and complexity reduction technique. It has become a popular and powerful tool for model reduction over the last decade, and has been successfully extended to parametrized linear dynamical systems \cite{morIonA14}, structured linearized systems~\cite{schulze2018data}, bilinear~\cite{morAntGI16} and quadratic-bilinear systems~\cite{morGosA18}, and has been formulated to work with time-domain data as well~\cite{morPehGW17}.

The principal feature of the \emph{interpolatory Loewner framework} is the construction of a rational interpolant matching the given sampled transfer function data.  This provides immediately a realization of a reduced model that matches the system response at  specified driving frequencies (or equivalently, at specified complex interpolation points).  The quality of the final reduced model is strongly tied to the choice of interpolation points and in the MIMO setting, also to the choice of tangent directions. If some additional flexibility exists in where the system may be sampled, $\mathcal{H}_2$-optimal reduced order interpolants may be feasibly computed \cite{BeattieGugercin2012RlznIndH2Approx}. For a detailed account of theoretical developments as well as implementation details  for Loewner-based methods relating especially to rational interpolation, we refer to \cite{AntBG20,ALI17}, and references therein.

Although our \QBT framework makes use of matrices that are, at least in many cases of interest, diagonally scaled Loewner and shifted Loewner matrices, rational interpolation plays no role for us.   Instead, we have observed that there are strong links among (a) quadrature rules \eqref{quad_P} and \eqref{quad_Q},  (b) the (limiting) implicit Gramian factorizations described in \eqref{exactGramFact}, and (c) the related Loewner / shifted Loewner factorizations.  These links couple together the features developed in Propositions \ref{prop:quadL}, \ref{prop:QuadError},  and \ref{prop:quadM}, and allow us ultimately to mimic \BT in a purely data-driven manner.

\subsubsection{Numerical examples} 

We use the [\textsf{heat}] and [\textsf{iss1r}] models
from \S\S\ref{sec:exmphankel} anda \ref{sec:firstcomp}
as test cases. \QBT is implemented using only the [\textsf{ExpTrap}] quadrature.

We choose 120 logarithmically-spaced points in the interval $[10^{-3}, 10^3] \imunit$ for the [\textsf{heat}] model and 400 logarithmically-spaced points in the interval $[10^{-1}, 10^2] \imunit$ for the [\textsf{iss1r}] model. 
{In the interpolatory Loewner framework, the singular values of 
	$\origLL$ and/or the augmented matrix $[~\origLL~~\origMM~]$ play a fundamental role in determining reduction order,  analogous to Hankel singular values in \BT; see 
	\cite[\S 4.3]{AntBG20}. Therefore, in addition to the \QBT quantities,  we construct $\origLL$ and $\origMM$ corresponding to the same transfer function samples and compute (i) the true 
	Hankel singular values, (ii)  \QBT-based Hankel singular values, and 
	(iii) the singular values of $\origLL$ and $[~\origLL~~\origMM~]$.
	Results are depicted in Fig. \ref{fig:4}. While the Hankel singular values computed via [\textsf{ExpTrap}]
	approximate the true ones fairly well, the singular values corresponding to the Loewner method appear to follow the trend, only up to a scaling factor. This simple example shows the importance of 
	the quadrature weights appearing in the \QBT framework.}
\begin{figure}[ht]
	\hspace{-6mm}
	\includegraphics[scale=0.205]{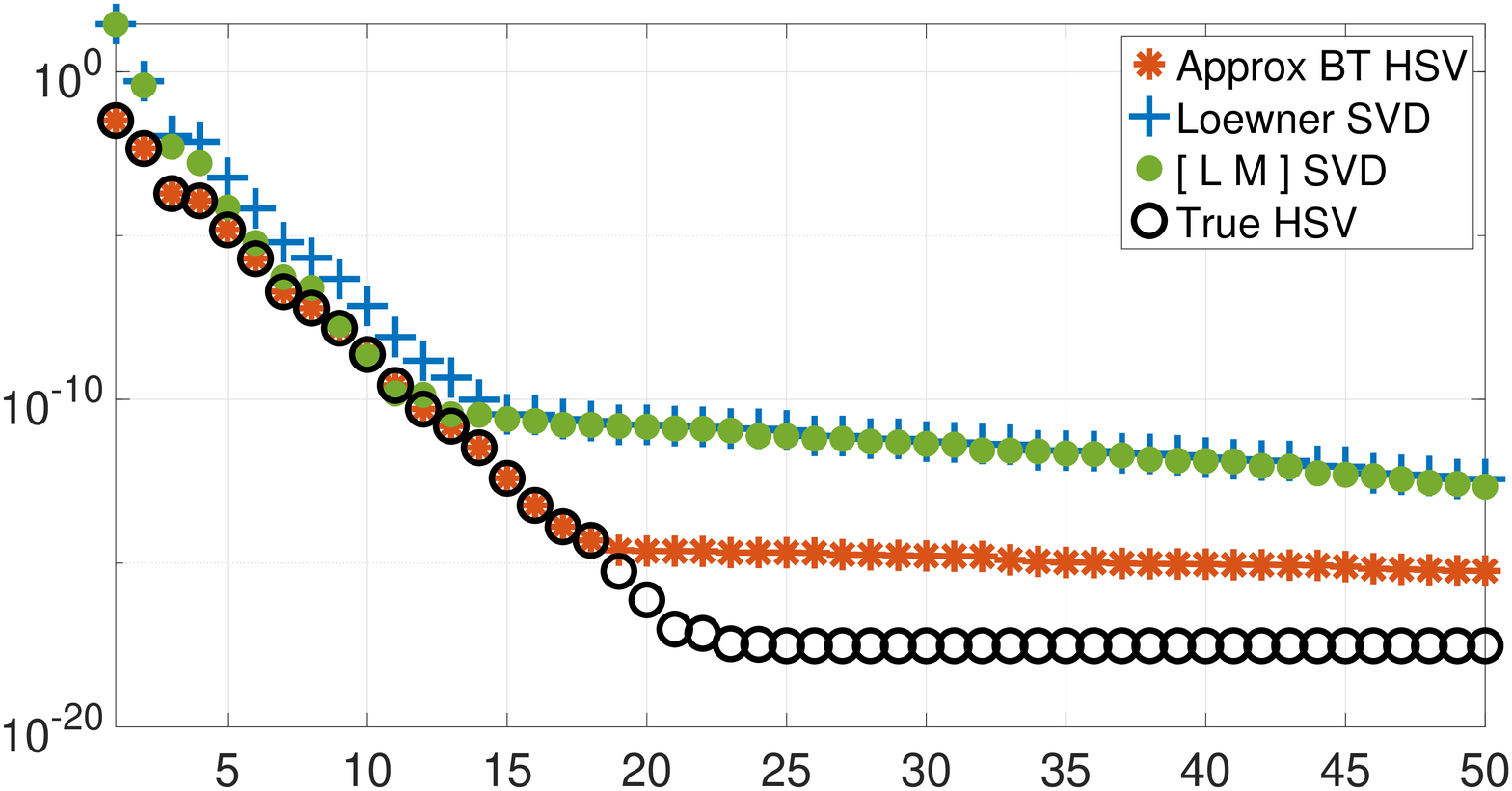}
	\hspace{-7mm}
	\includegraphics[scale=0.205]{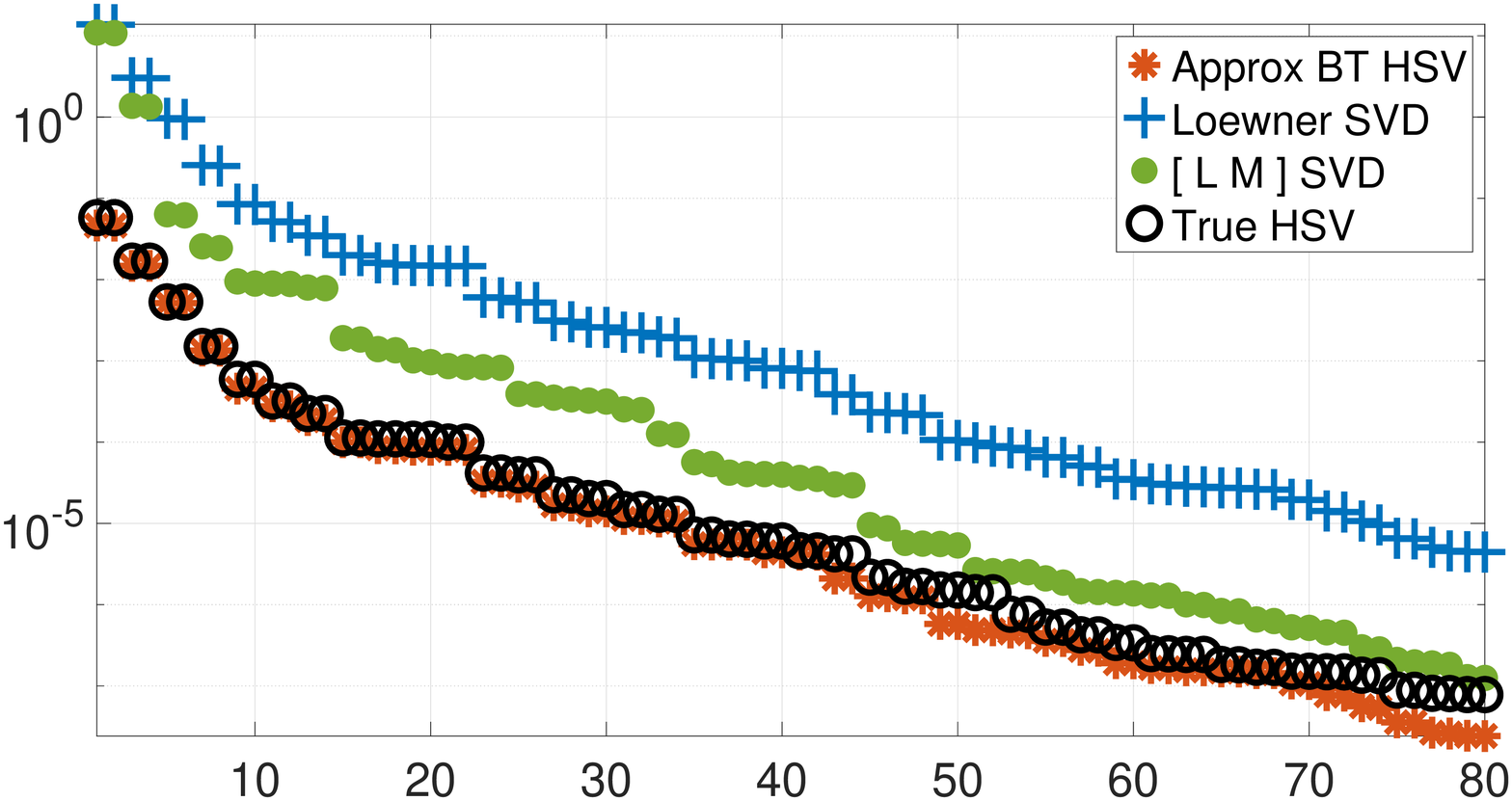}
	\vspace{-5mm}
	\caption{The computed singular values for the heat model (left) and the {\textrm [\textsf{iss1r}]} model (right).}
	\label{fig:4}
	\vspace{-2mm}
\end{figure}

Next, we consider the SISO version of the [\textsf{iss1r}] model
and apply both \QBT with [\textsf{ExpTrap}] using 200 logarithmically-spaced points in the interval $[10^{-1}, 10^2] \imunit$ and  the Loewner-based interpolation method. We  vary the reduction order from $r=2$ to $r=24$, and collect the $\cH_\infty$  and  the $\cH_2$ norms of the error systems for both methods. The results are presented in Fig. \ref{fig:5} where the label \textsf{[Loewner-200]} refers to the Loewner method with $N = 200$ points.  Fig. \ref{fig:5} shows that 
\emph{for the same transfer function data},  it pays off to use appropriate quadrature weights  to perform \QBT as 
\QBT outperforms the Loewner approach and mimics \BT much more closely.

For the two quadrature rules that we consider here (described in Appendix A), the quadrature nodes are closely related to spectral/pseudospectral approximations of the integrand and so will also be effective for producing accurate interpolants for interpolation-based methods including those that typify the Loewner model reduction framework. Other choices for quadrature rules (e.g., Gauss  rules) might not yield such effective interpolants. 
\begin{figure}[h]
	\hspace{-6mm}
	\includegraphics[scale=0.205]{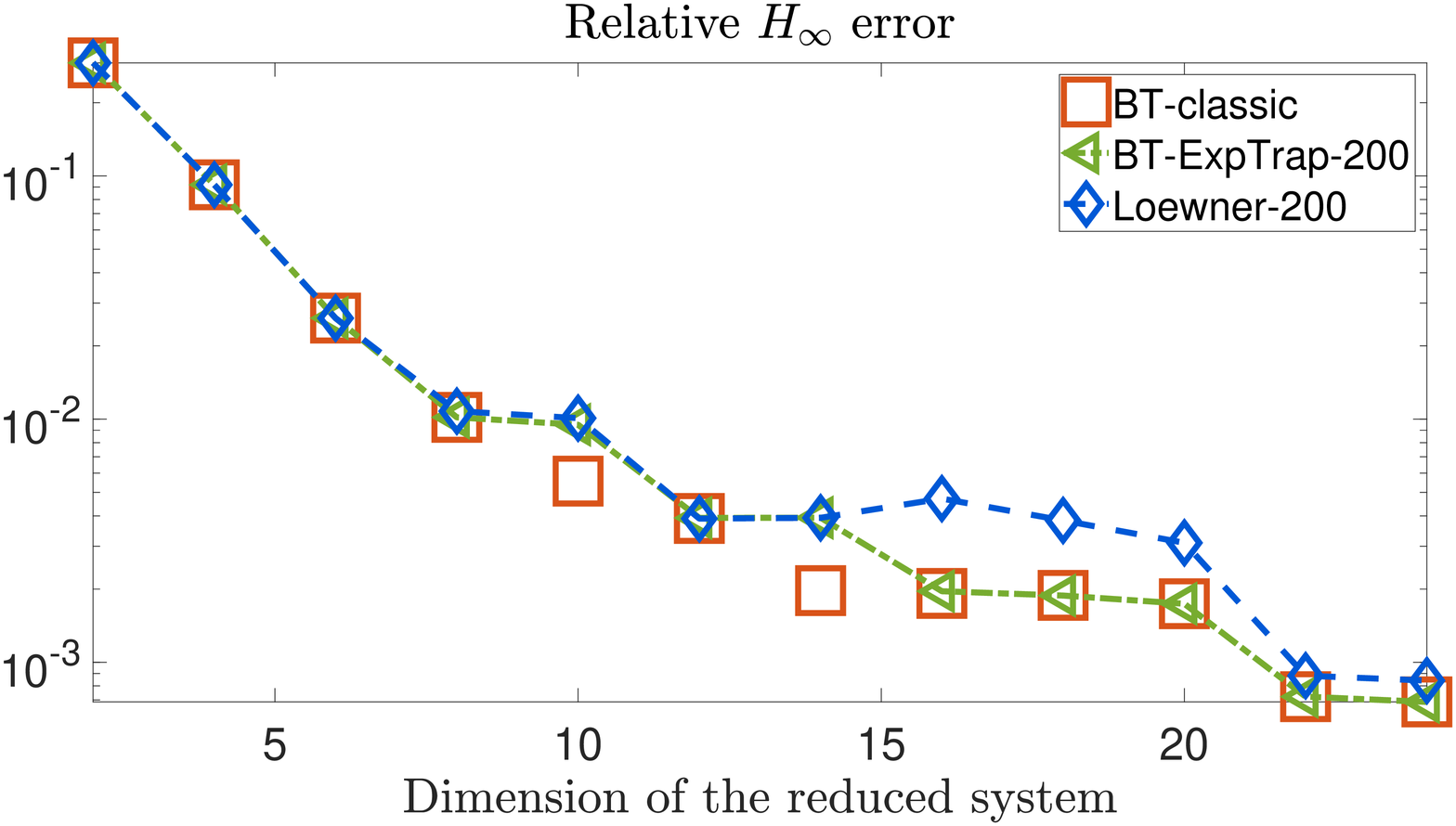}
	\hspace{-7mm}
	\includegraphics[scale=0.205]{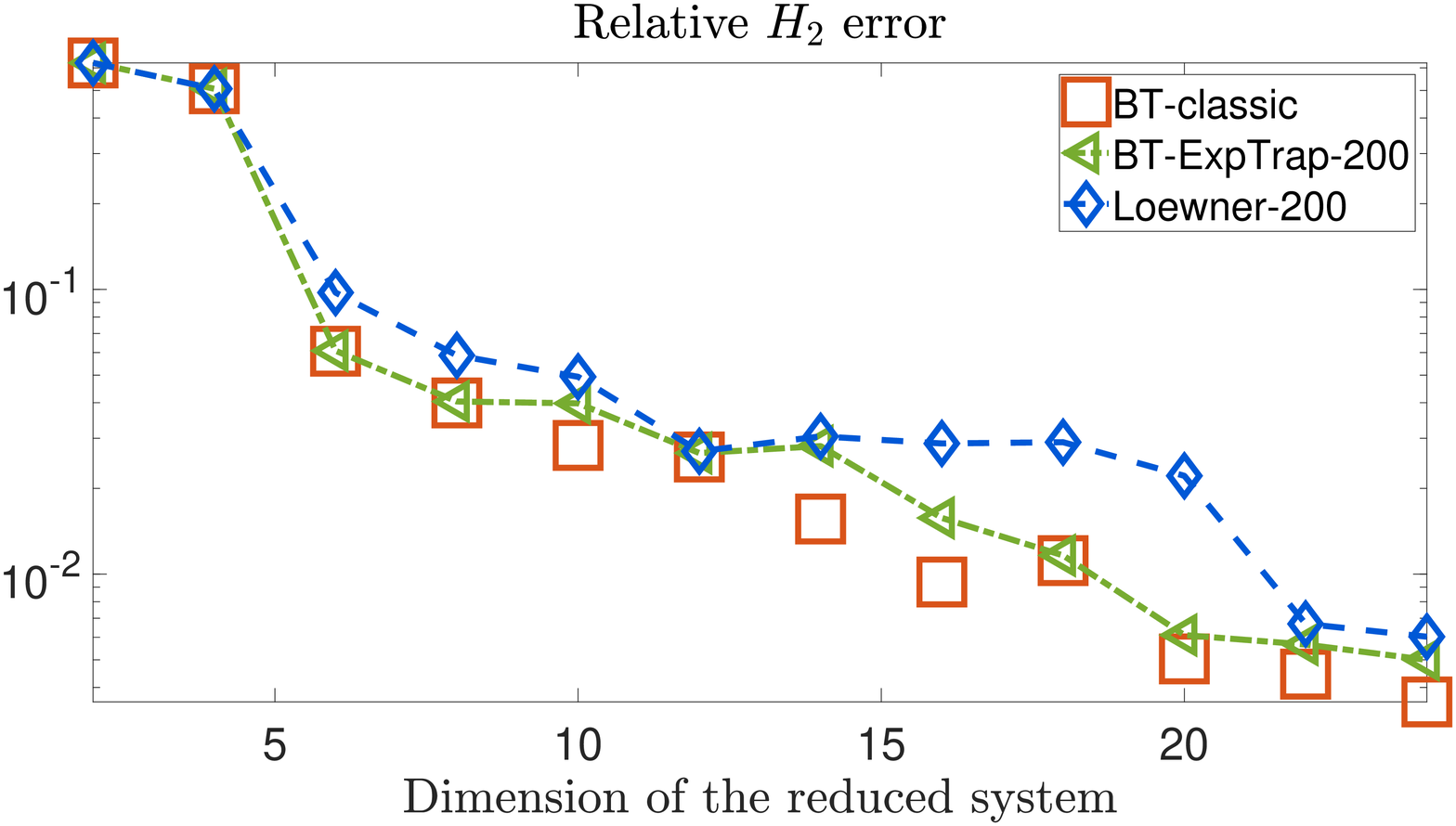}
	\vspace{-2mm}
	\caption{Relative $\cH_\infty$ (left) and  $\cH_2$ (right) approximation errors for different methods applied to the {\textrm [\textsf{iss1r}]} model.}
	\label{fig:5}
	\vspace{-3mm}
\end{figure}

\section{Refinements and Extensions}
\label{sec:ref_and_ext}

\subsection{Keeping it real} \label{sec:realness}

The original system has been presumed to be real, insofar as real-valued inputs and real-valued initial states assure real-valued outputs.   This guarantees in turn that a \emph{real}  realization involving matrices $\bE,\,\bA \in \mathbb{R}^{n \times n}, \ \bB\in \mathbb{R}^{n \times m}, \bC\in \mathbb{R}^{p \times n}$ is possible, though not necessarily explicitly available. 
Equivalently, the associated system transfer function $\mathsf{H}(s)=\bC(s\bE-\bA)^{-1}\bB$ displays a conjugate symmetry of system poles and residues (i.e., complex poles and residues must occur in conjugate pairs).  If we utilize \emph{symmetric quadrature rules} in the approximations given in \eqref{quad_P} and \eqref{quad_Q} --- that is, if the sets of quadrature nodes and associated weights, $\{(\zeta_j,\,\rho_j)\}_{j=1}^{\np}$ and $\{(\omega_k,\,\phi_k\}_{k=1}^{\nq}$,
are symmetrically distributed on the real axis with respect to $0$ ---
then we will be able to retain this conjugate symmetry, 
accomplish  computations equivalent to Steps 2, 3, and 4 of \QBT using only real arithmetic, and be assured that the final reduced model will inherit the conjugate symmetry of the original system, thus representing a reduced-order real dynamical system. 

For simplicity, we focus on the SISO case and suppose that both node sets have an even number of points that are symmetrically distributed across the real axis, say $\np=2\nu_p$ and $\nq=2\nu_q$.
(More general settings are considered in \cite{AntBG20}, p. 203). Relabel these node sets as 
$$
\begin{array}{c}
\zeta_{-\nu_p}< \zeta_{-\nu_p+1}< \cdots < \zeta_{-1} < 0 < \zeta_{1}< \cdots < \zeta_{\nu_p-1}<\zeta_{\nu_p}\\[2mm]
\omega_{-\nu_q}< \omega_{-\nu_q+1}< \cdots < \omega_{-1} < 0 < \omega_{1}< \cdots < \omega_{\nu_q-1}<\omega_{\nu_q},
\end{array}
$$
noting that $\zeta_{-j}=-\zeta_{j}$ and with corresponding weights, $\rho_{-j}=\rho_j$, for $j=1,...,\nu_p$.  Similarly, we have $\omega_{-k}=-\omega_{k}$ and $\phi_{-k}=\phi_{k}$
for $k=1,...,\nu_q$. The nodes/weights are then reordered with respect to this $\pm$ node pairing: 
{
	\begin{equation}\label{RealReord}
	\{\,\zeta_{1},\, \zeta_{-1},\,\zeta_{2},\, \zeta_{-2}, \ldots,\,\zeta_{\nu_p},\,\zeta_{-\nu_p} \} 
	\quad \mbox{and} \quad 
	\{\,\omega_{1},\, \omega_{-1},\,\omega_{2},\, \omega_{-2}, \ldots,\,\omega_{\nu_q},\,\omega_{-\nu_q} \}.
	\end{equation} }
When the samples obtained in Step 1 of \QBT are reordered in the same way, they appear now as a sequence of conjugate pairs: 
$$
\{H(\imunit \zeta_j), \overline{H(\imunit \zeta_j)} \}_{j=1}^{\nu_p}\quad\mbox{ and }\quad \{H(\imunit \omega_k), \overline{H(\imunit \omega_k)} \}_{k=1}^{\nu_q}.
$$ Notice that one need only explicitly evaluate the first function of each pair.

Now instead of \eqref{quadLoew},
we reconsider 
the construction of $\quadLL$
taking into account the new node ordering given in \eqref{RealReord}.  
With respect to this node ordering, $\quadLL$ can be partitioned into $2\times 2$ blocks, $\quadLL_{k,j}^{(2)}$, conforming to the conjugate pairing of quadrature nodes in \eqref{RealReord} --- left data at $(\imunit \omega_k, -\imunit \omega_k)$; 
right data at 
$(\imunit \zeta_j, -\imunit \zeta_j)$:
$$
\quadLL_{k,j}^{(2)}=
- \phi_k \rho_j\left[ 
\begin{array}{cc}
\frac{H(\imunit\,\omega_k) - H(\imunit\,\zeta_j)}{\imunit\,\omega_k - \imunit\, \zeta_j} & \frac{H(\imunit\,\omega_k) - \overline{H(\imunit\,\zeta_j)}}{\imunit\,\omega_k -(-\imunit\, \zeta_j)} \\[2mm]
\frac{\overline{H(\imunit\,\omega_k)} - H(\imunit\,\zeta_j)}{(-\imunit\,\omega_k) - \imunit\, \zeta_j} & \frac{\overline{H(\imunit\,\omega_k)} - \overline{H(\imunit\,\zeta_j)}}{(-\imunit\,\omega_k) -(-\imunit\, \zeta_j)} \\[2mm]
\end{array}\right],
$$
for $j=1,...,\nu_p$ and $k=1,...,\nu_q$. Note that each $\quadLL_{k,j}^{(2)}$ satisfies a Sylvester equation: 
\small
\begin{align*}
&\quadLL_{k,j}^{(2)}\,\left[\begin{array}{cc}
\imunit\,\omega_k & 0 \\[2mm] 0 & -\imunit\,\omega_k
\end{array}\right] - 
\left[\begin{array}{cc}
\imunit\,\zeta_j & 0 \\[2mm] 0 & -\imunit\,\zeta_j
\end{array}\right]\,\quadLL_{k,j}^{(2)} \\ &= \phi_k \rho_j\left(\left[\begin{array}{c}
1 \\[1mm] 1
\end{array}\right]\, [\,H(\imunit\,\zeta_j)\quad
\overline{H(\imunit\,\zeta_j)}\,]-
\left[\begin{array}{c}
H(\imunit\,\omega_k)\\[2mm]
\overline{H(\imunit\,\omega_k)}
\end{array}\right]\, [\,1\quad 1\,]\right).
\end{align*}
\normalsize
This uniquely determines each $\quadLL_{k,j}^{(2)}$ (and hence, $\quadLL$) when the sets of quadrature nodes, 
$\{\zeta_j\}$ and $\{\omega_k\}$, are disjoint as we have assumed. 

$\quadLL$ is unitarily equivalent to a \emph{real} matrix as we next demonstrate explicitly. 
Define $\bJ={\scriptsize \frac{1}{\sqrt{2}}\left[\begin{array}{cc} 1 & -\imunit \\ 1 & \phantom{-} \imunit\end{array}\right]}$
and consider a unitary equivalence defined by
$$
\quadLL^R = (\bI_{\nu_q}\otimes \bJ^{\star})\ \quadLL\ (\bI_{\nu_p}\otimes \bJ).
$$
$\quadLL^R$ can be partitioned into $2\times 2$ blocks, conforming precisely to the earlier partitioning of $\quadLL$, and indeed one finds 
for each block, $\quadLL_{k,j}^{(2)R}=\bJ^{\star}\quadLL_{k,j}^{(2)}\bJ$
for $j=1,...,\nu_p$ and $k=1,...,\nu_q$.
One may check directly that $\quadLL_{k,j}^{(2)R}$ is real.  Indeed, each block may be determined uniquely from a corresponding (real) Sylvester equation:
\footnotesize
\begin{align*}
\quadLL_{k,j}^{(2)R}\,\left[\begin{array}{cc}
0 & \omega_k \\[2mm] -\omega_k & 0
\end{array}\right] - 
\left[\begin{array}{cc}
0 & \zeta_j \\[2mm] -\zeta_j & 0
\end{array}\right]\,\quadLL_{k,j}^{(2)R} 
= 2\phi_k \rho_j
\left[\begin{array}{cc}
\mathsf{Re}\left(H(\imunit\,\zeta_j) - H(\imunit\,\omega_k)\right) \ 
&\  \mathsf{Im}\left(H(\imunit\,\zeta_j)\right) \\[2mm]
\mathsf{Im}\left(H(\imunit\,\omega_k)\right) & 0
\end{array}\right],
\end{align*}
\normalsize
for $j=1,...,\nu_p$ and $k=1,...,\nu_q$. 
One may solve these Sylvester equations for $\quadLL_{k,j}^{(2)R}$, in real arithmetic, independently of one another in parallel. 

Similar considerations may be applied to \eqref{quadshiftLoew}, \eqref{quadh}, and \eqref{quadg},  which define respectively, the complementary matrix $\quadMM$, and vectors $\quadLb$ and $\quadcU$.  For $\quadMM$, we find as before, a unitarily equivalent real matrix:
$$
\quadMM^R = (\bI_{\nu_q}\otimes \bJ^{\star})\ \quadMM\ (\bI_{\nu_p}\otimes \bJ),
$$
which may be partitioned into submatrices,
$\quadMM_{k,j}^{(2)R}=\bJ^{\star}\quadMM_{k,j}^{(2)}\bJ$
for $j=1,...,\nu_p$ and $k=1,...,\nu_q$, conforming to the partitioning for $\quadLL^R$. These submatrices are directly computable from real Sylvester equations:
\begin{align*}
&\quadMM_{k,j}^{(2)R}\,\left[\begin{array}{cc}
0 & \omega_k \\[2mm] -\omega_k & 0
\end{array}\right] - 
\left[\begin{array}{cc}
0 & \zeta_j \\[2mm] -\zeta_j & 0
\end{array}\right]\,\quadMM_{k,j}^{(2)R} \\ &= 2\phi_k \rho_j
\left[\begin{array}{cc}
\mathsf{Im}\left(\omega_k\,H(\imunit\,\omega_k)-\zeta_j\,H(\imunit\,\zeta_j)\right) \ 
&\  \mathsf{Re}\left(\zeta_j\,H(\imunit\,\zeta_j)\right) \\[2mm]
\mathsf{Re}\left(\omega_k\,H(\imunit\,\omega_k)\right) & 0
\end{array}\right],
\end{align*}
for $j=1,...,\nu_p$, \ $k=1,...,\nu_q$.
For $\quadLb$ and $\quadcU$, we may define corresponding real vectors, 
$$
\quadLb^{(R)} = (\bI_{\nu_q}\otimes \bJ^{\star})\quadLb \quad  \mbox{and} \quad \quadcU^{(R)T}= \quadcU^T (\bI_{\nu_p}\otimes \bJ).
$$
If we return to \QBT (Algorithm \ref{quadbt}) and replace $\quadLL$, $\quadMM$, $\quadLb$, and $\quadcU$ with their real counterparts derived above, $\quadLL^{(R)}$, $\quadMM^{(R)}$, $\quadLb^{(R)}$, and $\quadcU^{(R)}$, then the computations in Steps (3) and (4) may be done entirely in real arithmetic, and the final reduced model defined in Step (4) will be real, even in the presence of rounding errors.  

\subsection{MIMO Systems}\label{MIMOsys}
We have thus far presented an analysis framework specific to SISO dynamical systems, however the MIMO case follows from this nearly immediately; the only change occurs in the construction of 
$\quadLL$, $\quadMM$, $\quadLb$ and $\quadcr$, and not in  
Algorithm \ref{quadbt} itself.   Specifically, assume that the underlying dynamical system has $m$ inputs and $p$ outputs (e.g., as in \eqref{OrigSys}). In this case, the transfer function function is a matrix-valued rational function and we will denote it by $\bH(s)$ so that for any frequency $\hat{s} \in \mathbb{C}$, we have $\bH(\hat{s}) \in \mathbb{C}^{p \times m}$. Therefore, assuming the same set-up as in Proposition 
\ref{prop:quadL},
in defining $\quadLL$, we have
\begin{equation} \label{quadLoewmimo}
\quadLL_{k,j} = - \phi_k \rho_j \frac{\bH(\imunit\omega_k) - \bH(\imunit\zeta_j)}{\imunit\omega_k - \imunit \zeta_j} \in \mathbb{C}^{p\times m},
\end{equation}
as the $(k,j)$th \emph{block} of $\quadLL$ and therefore $\quadLL \in \mathbb{C}^{(p \cdot \nq)\times (m \cdot \np)}$. Similarly, in Proposition 
\ref{prop:quadM}, we will have
\begin{equation} \label{quadLoewmimo2}
\quadMM_{k,j} = - \phi_k \rho_j \frac{\imunit \omega_k\bH(\imunit\omega_k) - \imunit \zeta_j\bH(i\zeta_j)}{\imunit\omega_k - \imunit \zeta_j}\in \mathbb{C}^{p\times m},
\end{equation}
as the $(k,j)$th \emph{block} of $\quadMM$ and thus $\quadMM \in \mathbb{C}^{(p \cdot \nq)\times (m \cdot \np)}$. Since we are dealing with MIMO systems, we replace $\quadLb$ with $\quadLB=\quadL^* \bB
\in \mathbb{C}^{(\nq \cdot p) \times m}$ and $\quadcU^T$ with 
$\quadCU = \bC \quadU \in \mathbb{C}^{p \times (\np \cdot m)}$ with the $k$th row block of $\quadLB$ and $j$th columns block of $\quadCU$  given by 
$$
\quadLB_k = \phi_k \bH(\imunit \omega_k) \in \mathbb{C}^{p \times m} \quad \mbox{and} \quad 
\quadCU_j = \rho_j \bH(\imunit \zeta_j) \in \mathbb{C}^{p \times m}.
$$
Then, Algorithm \ref{quadbt} simply proceeds with these new definitions. 

\subsubsection{Numerical example: the MIMO case}

We examine performance for the [\textsf{iss1r}] numerical test case, which has $p=3$ outputs and $m=3$ inputs. Modifications specified above in \S\ref{MIMOsys} are followed to adapt \QBT to a MIMO setting.

Choose both 200 and 400 logarithmically-spaced points in the interval $[10^{-1}, 10^2] \imunit$  as nodes for [\textsf{ExpTrap}] (the left and right nodes are chosen to be distinct). For  [\textsf{B/CC}], choose 
$L=10$ and $L = 10.5$ for $\quadU$ and $\quadL$, respectively,
using both 200 and 400 nodes. We then vary the reduction order from $r=2$ to $r=24$, and compute the relative $\cH_\infty$ and $\cH_2$ errors, corresponding to both quadrature schemes. The results are presented in Fig. \ref{fig:6}. 
{As in the SISO case, \QBT accurately mimics \BT performance.
	In this MIMO example,  [\textsf{B/CC}] outperforms [\textsf{ExpTrap}].}

\begin{figure}[ht]
	\hspace{-6mm}
	\includegraphics[scale=0.205]{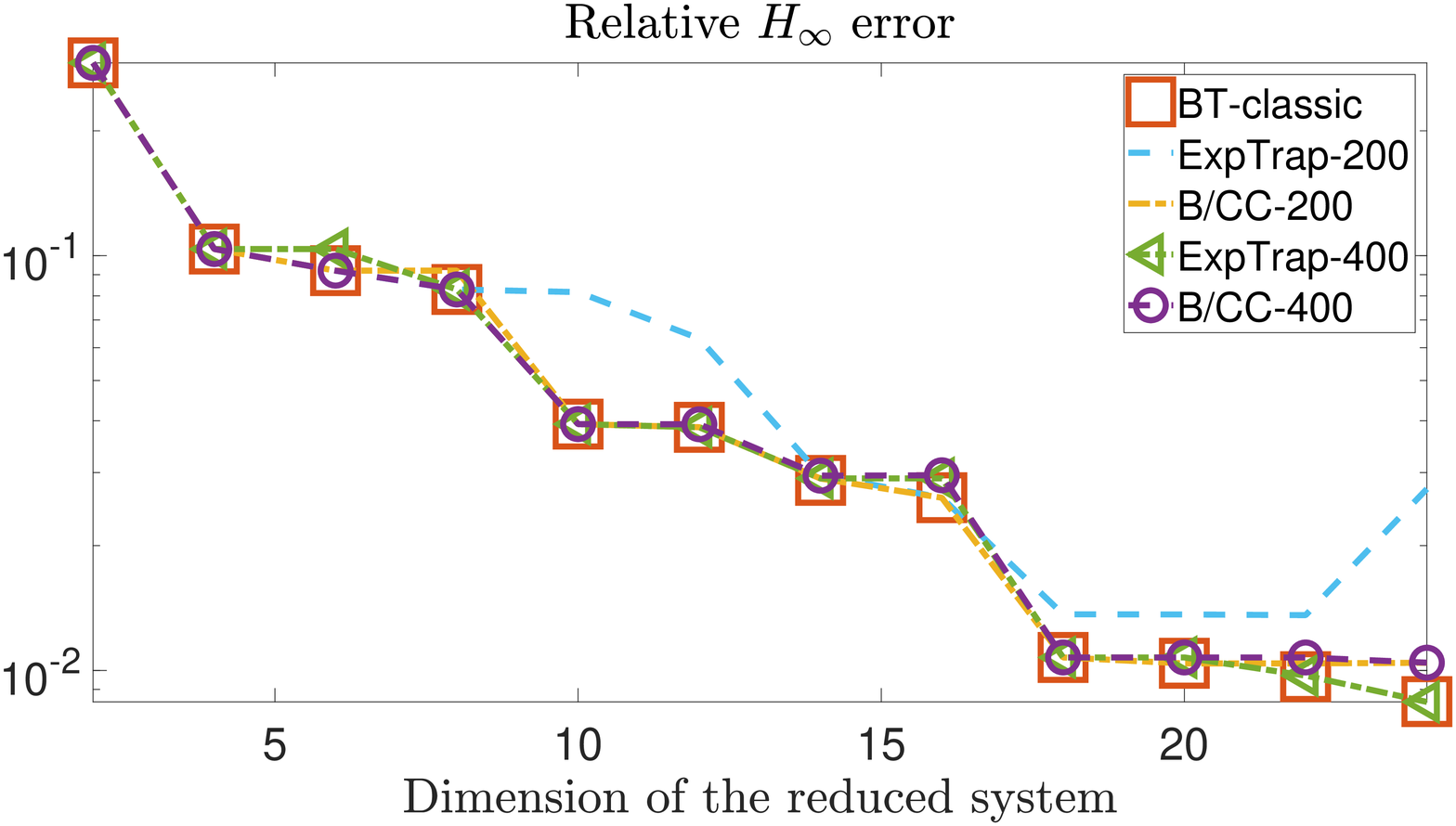}
	\hspace{-7mm}
	\includegraphics[scale=0.205]{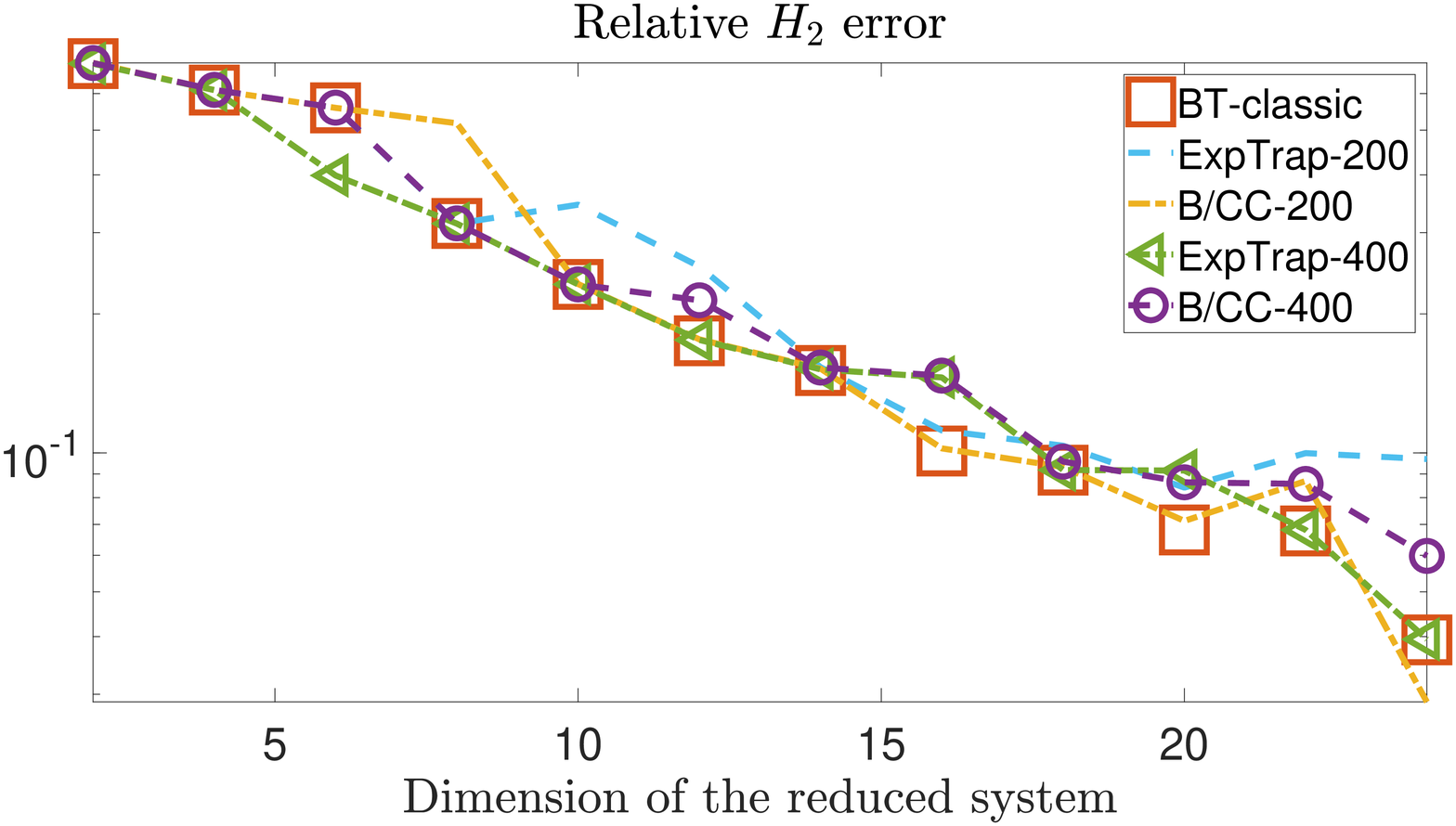}
	\vspace{-3mm}
	\caption{Relative $\cH_\infty$ (left) and $\cH_2$ (right)  approximation errors for different quadrature schemes ({\textrm[\textsf{ExpTrap}]} vs {\textrm[\textsf{B/CC}]}) and different numbers of nodes (200 vs 400) for {\textrm [\textsf{iss1r}]} MIMO model.}
	\label{fig:6}
	\vspace{-3mm}
\end{figure}

\subsection{Discrete-time systems}\label{DISCsys}

Consider a discrete-time dynamical system characterized by the  difference equations
\begin{equation} \label{OrigSys_disc}
\begin{array}{c}\bE \bx_{k+1} = \bA\, \bx_{k} + \bB\, \bu_k, \\[1mm] \by_k = \bC \bx_k,\end{array}
\end{equation}
where the input mapping is given by $\bu : \mathbb{Z} \rightarrow \mathbb{R}^m$, the state trajectory/variable is
$\bx: \mathbb{Z} \rightarrow \mathbb{R}^n$, and the output mapping is $y : \mathbb{Z} \rightarrow \mathbb{R}^p$. The system matrices, as in the continuous case, are given by $\bA, \bE \in \mathbb{R}^{n \times n}$, and  $\bB\in \mathbb{R}^{n\times m}$, $\bC \in \mathbb{R}^{p\times n}$. We assume that the discrete system is \emph{asymptotically stable}, i.e., the eigenvalues of the matrix pencil $(\bA,\bE)$  are located in the open unit disc.

The reachability Gramian $\bP$ in the frequency domain is given by
\begin{align} \label{gram_P_freq_disc}
\begin{split}
\bP =
\frac{1}{2\pi}\int_{0}^{2 \pi} (e^{\imunit \zeta}\bE -\bA)^{-1} \bB \bB^T (e^{-\imunit \zeta}\bE^T-\bA^T)^{-1} d \zeta.
\end{split}
\end{align}
Similarly, the observability Gramian can be represented as
\begin{align} \label{gram_Q_disc}
\begin{split}
\bQ = \frac{1}{2\pi}\int_{0}^{2 \pi} (e^{-\imunit \omega} \bE^T -\bA^T)^{-1} \bC^T \bC (e^{\imunit \omega}\bE -\bA)^{-1} d \omega
\end{split}
\end{align}
$\bP$ and $\bQ$ satisfy the following Stein equations:
\begin{align}  \label{lyapforPQ_disc}
\bA \bP\bA^T + \bB\bB^T = \bE \bP \bE^T, \ \
\bA^T \bQ \bA   + \bC^T \bC = \bE^T \bQ \bE.
\end{align}
As in the continuous-time case, we employ a numerical quadrature rule to approximate the  the Gramian $\bP$ defined in \eqref{gram_P_freq_disc} to obtain the approximant
\begin{align} \label{quad_P_disc}
\begin{split}
\bP \approx \quadP = \sum_{j=1}^{\np} \rho_j^2  (e^{\imunit \omega_j} \bE -\bA)^{-1} \bB \bB^T (e^{-\imunit \omega_j} \bE^T -\bA^T)^{-1}
\end{split}
\end{align}
with $\rho_j^2$ and $\omega_j \in [0,2\pi]$ denoting, respectively, quadrature weights and nodes.

We apply a similar quadrature approximation for $\bQ$ and obtain
\begin{align} \label{quad_Q_disc}
\bQ \approx \quadQ = \sum_{k=1}^{\nq} \phi_k^2  (e^{-\imunit \zeta_k} \bE^T -\bA^T)^{-1} \bC^T \bC (e^{\imunit \zeta_k} \bE -\bA)^{-1},
\end{align}
where $\phi_k^2$ and $\zeta_k \in [0,2\pi]$ denote, respectively, quadrature weights and nodes associated with approximating $\bQ$.
{The rest of the \QBT formulation for discrete systems follows immediately as in Algorithm \ref{quadbt}. We omit those details here for brevity. 
}

\subsubsection{Numerical example: the discrete-time case}

We analyze an $n$th-order lowpass digital Butterworth filter with normalized cutoff frequency $W_c$, using the parameter values $n=40$ and $W_c = 0.6$. (A realization can be obtained in \textsc{matlab} using the command \textsf{butter}.) 

Choose  {both 150 and 300} linearly-spaced points in the interval $[0, 2 \pi)$ as the $\omega_j$'s and $\zeta_k$'s in (\ref{quad_P_disc}) and (\ref{quad_Q_disc}), for  [\textsf{ExpTrap}]. Here, the left and right nodes are chosen to be the same. For the [\textsf{B/CC}] quadrature, the values of $\omega_j$'s and $\zeta_k$'s are chosen to be Chebyshev nodes instead; again both using 150 and 300 nodes. We  vary the reduction order
from $r=2$ to $r=40$ and collect the relative $\cH_\infty$ and $\cH_2$ errors, corresponding to both quadrature schemes and  classical \BT. 
The results are presented in Fig. \ref{fig:8}. 
{\QBT via both [\textsf{ExpTrap}] and
	[\textsf{B/CC}] almost exactly replicates the \BT behavior, illustrating the success of our proposed method in the discrete-time setting as well.
}

\begin{figure}[ht]
	\hspace{-6mm}
	\includegraphics[scale=0.205]{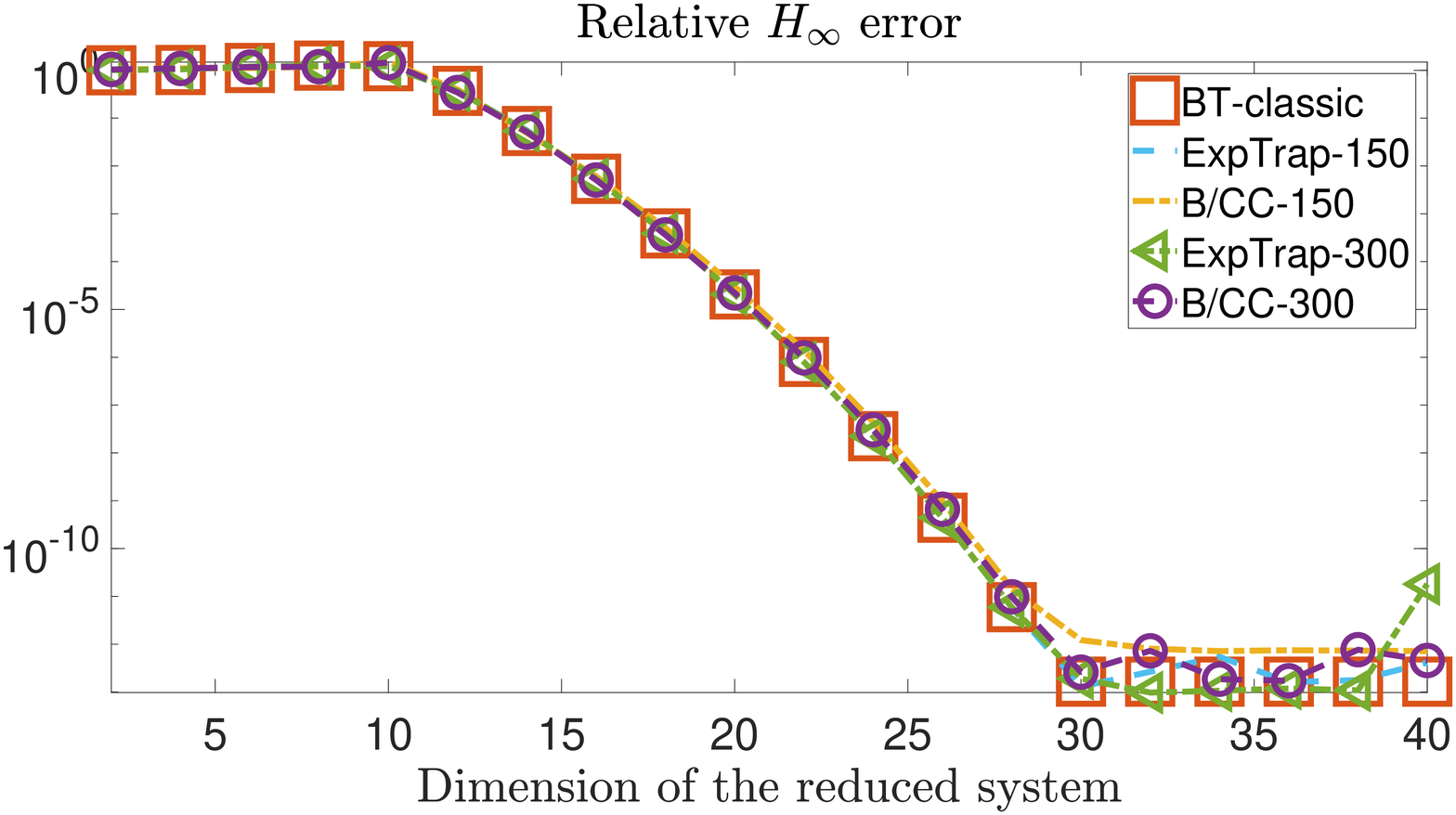}
	\hspace{-7mm}
	\includegraphics[scale=0.205]{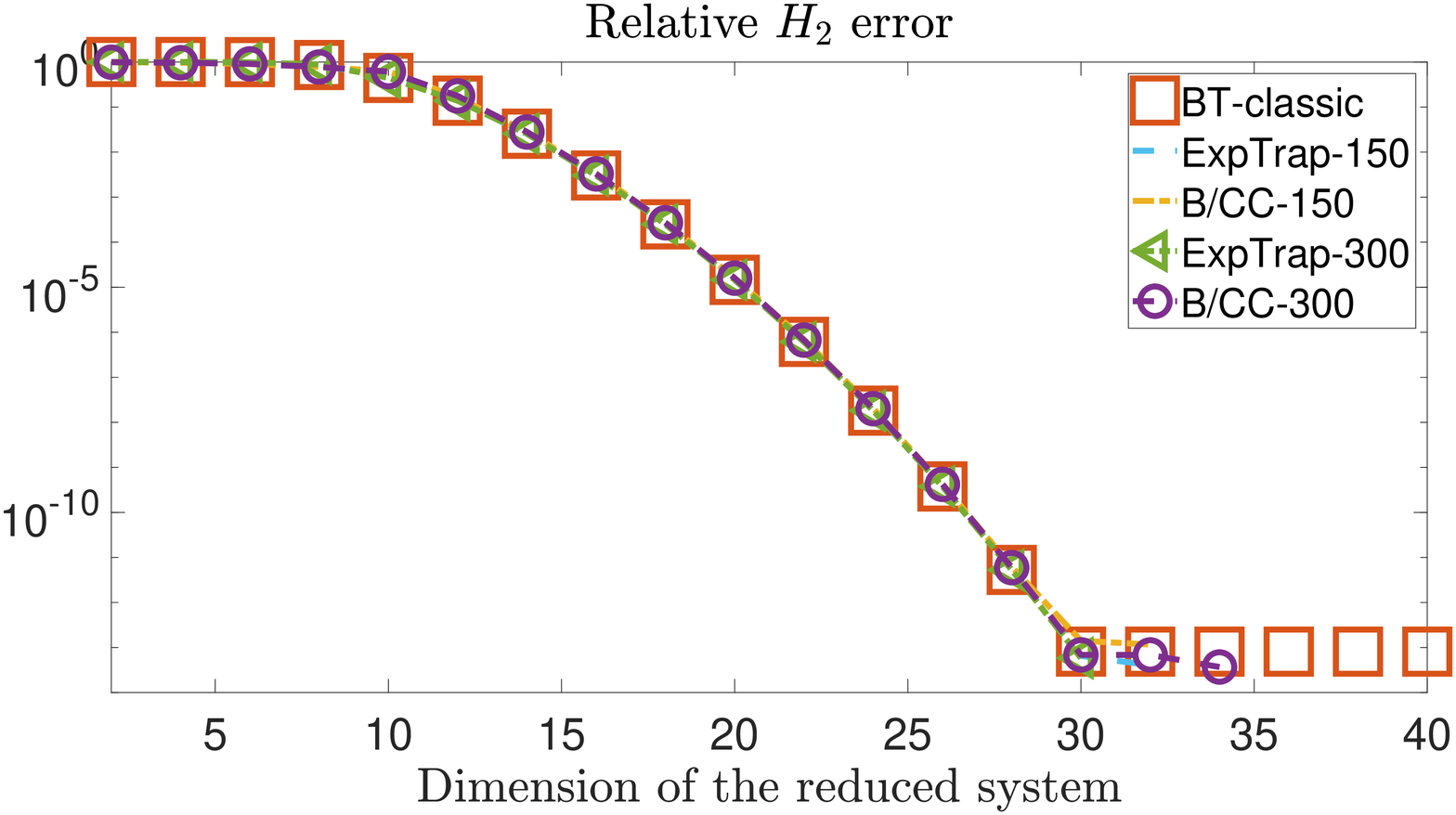}
	\vspace{-3mm}
	\caption{The relative $\cH_\infty$ approximation error for different quadratures applied to the Butterworth digital filter model.}
	\label{fig:8}
	\vspace{-2mm}
\end{figure}

\subsection{An infinite-dimensional example}\label{IRRsys}

{We revisit a numerical example from \cite{BeattieGugercin2012RlznIndH2Approx} 
	describing the response of an infinite-dimensional linear system, in this case a one-dimensional heat equation. The evolution of the temperature
	distribution on a semi-infinite rod is characterized by the PDE
	$\frac{\partial y}{\partial t}-\frac{\partial^2 y}{\partial z^2} = 0$, 
	where $y(z, t)$ denotes the temperature at position $z$ and time $t$. The temperature is controlled at the location $z = 0$, while the temperature at $z = 1$ is the quantity of interest. By enforcing particular initial and boundary conditions as in \cite{BeattieGugercin2012RlznIndH2Approx}, one may explicitly derive a
	transfer function:
	$ H(s) = \frac{Y(1,s)}{Y(0,s)} = e^{-\sqrt{s}}$,
	describing the mapping from $y(0, t)$ to $y(1; t)$ where $Y (\cdot, s)$ denotes the Laplace transform of $y(\cdot, t)$. 
	Balanced truncation can be approached within an operator-theoretic framework using ideas found for example in \cite{reis2014balancing,curtain1986}. 
	However, we are able to bypass these subtleties, since the transfer function $H(s)$ is immediately available, as well as its derivative $H'(s) = -\frac{e^{-\sqrt{s}}}{2\sqrt{s}}$. 
	
	We directly apply \QBT with [\textsf{ExpTrap}] using $50$ logarithmically-spaced points in the interval $[10^{-2}, 10^1] \imunit$. The results are displayed in Fig. \ref{fig:88}. In the left plot, we show the approximation error, $\mid H(\imunit \omega)- H_r(\imunit \omega)\mid$,  for $r=8,12,16,20$. In the right side of Fig. \ref{fig:88}, we present the approximate Hankel singular values of the transfer function $H(s) = e^{-\sqrt{s}}$. Reduced models via \QBT successfully approximate the original transfer function without requiring a discretization of the underlying PDE. 
}

\begin{figure}[ht]
	\hspace{-2mm}
	\includegraphics[scale=0.20]{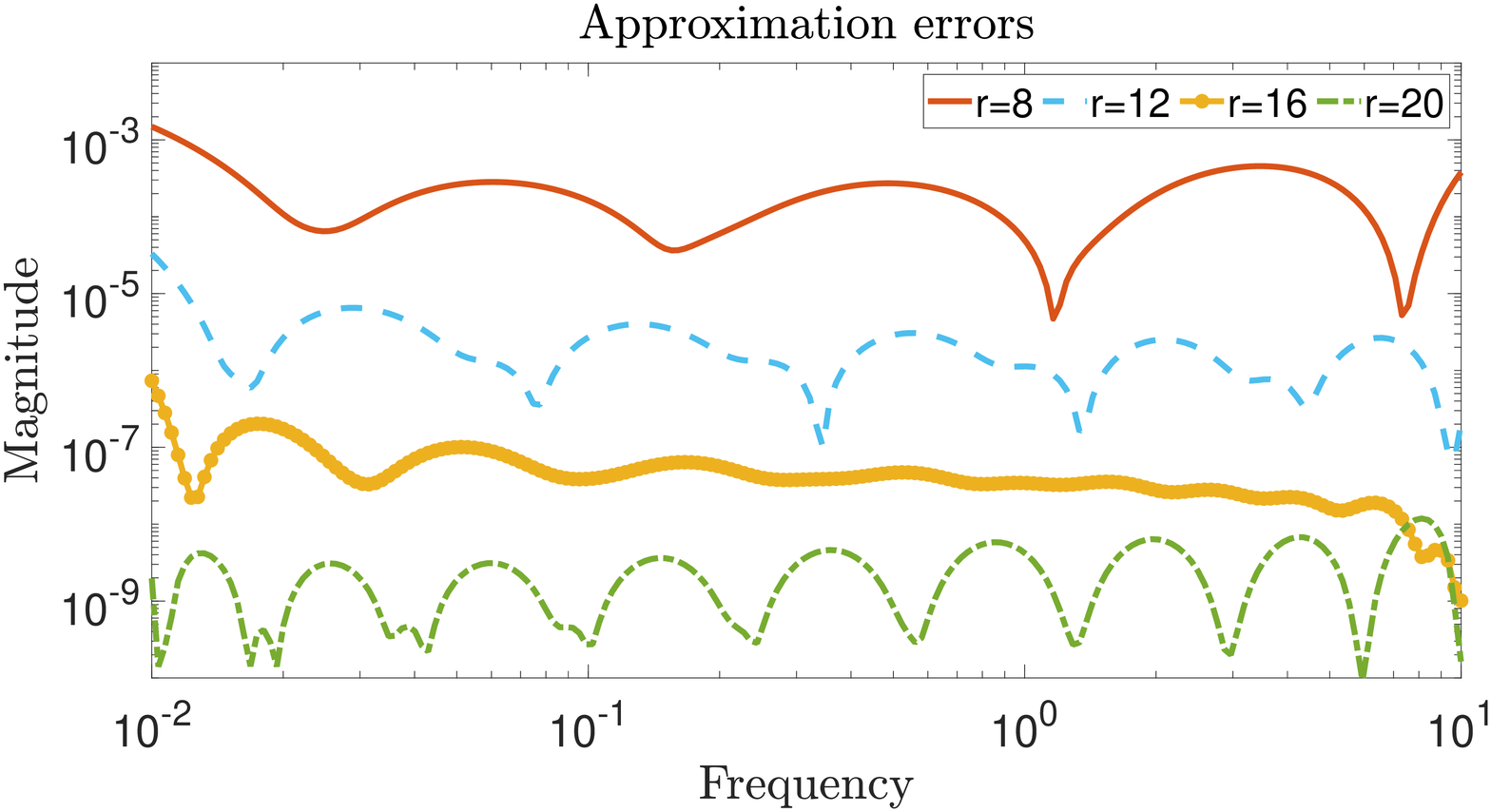}
	\hspace{-6mm}
	\includegraphics[scale=0.2]{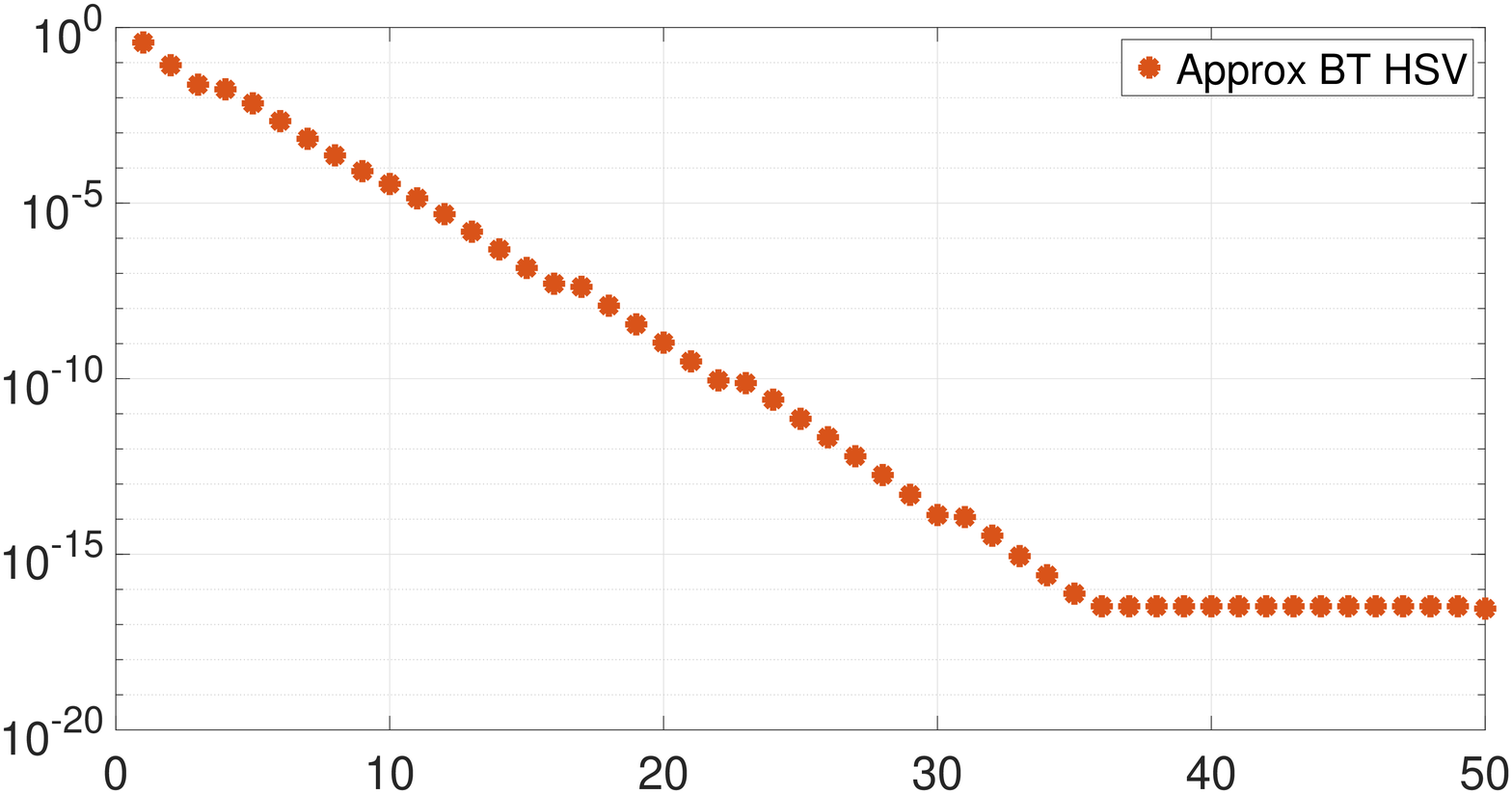}
	\vspace{-3mm}
	\caption{The approximation error for different reduction orders (left) and the approximated Hankel singular values (right) for $H(s) = e^{-\sqrt{s}}$.}
	\label{fig:88}
	\vspace{-2mm}
\end{figure}

\begin{remark}
	The numerical quadratures that we employ in our experiments converge exponentially fast with respect to the number of quadrature nodes used (see Appendix A).  Thus, one might not anticipate the need for a large number of quadrature nodes in order to ensure accuracy.   Nonetheless, exponential convergence  is an asymptotic property that might not be  readily observed in circumstances where the spectral abscissae of $\mathbf{A}$ is close to $0$, or if $\mathbf{A}$ is significantly nonnormal.   In such circumstances, the number of quadrature nodes that are required for accurate model reduction may grow quite large, and the cost of \QBT may become dominated by the cost of the singular value decomposition of $\quadLL$ in \eqref{SVD_BT2}.
	Indeed, with a large number, $N$, of quadrature nodes $N_q\approx N_p\approx N$, say, the cost of computing Step 2 of Algorithm \ref{quadbt} in order to extract an order $r$ reduced model, will grow at least as $\mathcal{O}(N^2 r)$, using standard approaches. 
	The computational burden of this step could be moderated potentially by using randomized algorithms, e.g., rSVD approaches \cite{HMT11}, and such approaches may not require
	that $\quadLL$ be explicitly available as well. 
\end{remark}

\section{{Time-domain formulation}} 
\label{sec:timeQBT}

Thus far, we have developed ideas that are central to \QBT presuming access to \emph{frequency-domain data}, i.e., data resulting from transfer function sampling at selected points in the complex plane. This formulation utilized a frequency domain description of $\bP$ and $\bQ$, expressed in \eqref{gram_P_freq} and \eqref{gram_Q}. In this section, we establish an analogous  framework in the time domain, sampling the system response at selected instants in time and
then using a time-domain representation of $\bP$ and $\bQ$ to derive a realization corresponding to an approximate balanced truncation. There are elements in common with time-domain \textsf{Balanced POD} \cite{WP02,Ro05,singler2009proper}, which uses time-domain state snapshots. We will show that, as in the frequency domain version, those implementations can be performed without access to state (internal) data. In our presentation, we will adapt the notation from \cite{Ro05} to our formulation. We will  focus on the single-input/single-output case for simplicity. However, unlike in \cite{Ro05}, we will include the $\bE$ term in our formulation to keep the paper consistent with the earlier frequency-domain framework of \QBT.

\subsection{Time-domain \QBT for continuous-time systems} \label{timeQBTcont}

As illustrated in \cite{WP02,Ro05,singler2009proper}, the \textsf{Balanced POD} approach can be applied in the time-domain by considering a numerical quadrature approximation to time-domain definitions of the Gramians, as given, e.g., in \eqref{gram_P} for the reachability Gramian. 
We also note that in the seminal paper \cite{moore1981principal} where \BT was originally introduced, Moore had already mentioned the potential of quadrature-based approximation to Gramians, thus motivating time-domain \textsf{Balanced POD} via empirical Gramians.
{We refer to the recent work \cite{morHim21} for a comparative study on
	\BT-related methods employing empirical Gramians.}
Following on to themes developed here, we will illustrate that the time-domain \textsf{Balanced POD} formulation can be fully implemented using only input/output data and in particular, without access to state variables or internally characterized state realizations. 

Assume that the time domain formula \eqref{gram_P} for $\bP$ is approximated via a numerical quadrature on $(0,\infty)$: 
{\small 
	\begin{align} 
	\begin{split}
	\bP = \int_{0}^{\infty} e^{\bE^{-1}\bA t} \bE^{-1}\bb\bb^T\bE^{-T} e^{\bA^T\bE^{-T} t} dt
	\approx 
	\sum_{j=1}^{\np} \rho_j^2 e^{\bE^{-1}\bA t_j} \bE^{-1}\bb \bb^T\bE^{-T} e^{\bA^T\bE^{-T} t_j}, 
	\end{split}
	\end{align}}\noindent where $( \rho_j^2,t_j)$ for $j=1,\ldots,\np$ denote quadrature weight/node pairs on $(0,\infty)$.  Define
\begin{equation}  \label{quadUtime}
\quadU = \left[ \,
\rho_1  e^{\bE^{-1}\bA t_1} \bE^{-1} \bb \quad 
\cdots \quad \rho_\np  e^{\bE^{-1}\bA t_\np} \bE^{-1} \bb \right] \in 
\IR^{n \times \np}.
\end{equation}
Then, $\quadP = \quadU \quadU^*$ is an approximation to $\bP$ with $\quadU$ as the 
(approximate) square-root factor. As pointed out in \cite{Ro05},  the columns of $\quadU$ without the quadrature weights can be viewed as the snapshot data obtained from simulating the system 
\begin{equation} \label{eq:primal}
\bE \dot{\bw} = \bA \bw,\quad\mbox{with}\quad\bw(0) = \mathbf{E}^{-1}\bb. 
\end{equation}
One may form a time-domain quadrature approximation to  $\bQ$ in a 
similar fashion:
{\small
	\begin{align} 
	\begin{split}
	\bQ = \int_{0}^{\infty} e^{\bE^{-T}\bA^T t} \bE^{-T}\bc\bc^T\bE^{-1} e^{\bA \bE^{-1} t} dt
	\approx 
	\sum_{i=1}^{\nq} \phi_i^2 e^{\bE^{-T}\bA^T \tau_i} \bE^{-T}\bc \bc^T\bE^{-1} e^{\bA\bE^{-1} \tau_i} 
	\end{split}
	\end{align}}\noindent with $(\phi_i^2,\tau_j)$ for $i=1,\ldots,\nq$ being quadrature weight/node pairs on $(0,\infty)$. 

\noindent
Evidently, the corresponding square-root factor is
\begin{equation}  \label{quadLtime}
\quadL^T = \left[ \begin{matrix}
\phi_1 \bc^T \bE^{-1} e^{\bA\bE^{-1} \tau_1} \\
\vdots \\ \phi_\nq \bc^T \bE^{-1} e^{\bA \bE^{-1} \tau_\nq}
\end{matrix}  \right] \in \mathbb{R}^{\nq \times n},
\end{equation}
so that $\quadQ = \quadL \quadL^*$ approximates $\bQ$. 
Similar to $\quadU$, the columns of $\quadL$ without quadrature weights can be viewed as the snapshot data from simulating the dual system given by
\begin{equation} \label{eq:dual}
\bE^T \dot{\bz} = \bA^T \bz,\quad\mbox{with}\quad\bz(0) = \mathbf{E}^{-T}\bc. 
\end{equation}

The square-root implementation of time-domain \textsf{Balanced POD} \cite{Ro05} proceeds by using $\quadU$ and $\quadL$ in place of $\bU$ and $\bL$ in the square-root implementation of \BT in Algorithm \ref{origbt}. So far, this still reflects a realization-\emph{dependent} implementation, i.e., we require access to state-space representation.  We will avoid this requirement below.

As we did with frequency-domain data, we inspect 
$\quadLL = \quadL^T \bE \quadU \in \IR^{\nq \times \np}$. 
Using \eqref{quadUtime} and \eqref{quadLtime},
the $(i,j)$th component of $\quadLL$ is obtained as 
\begin{align} \label{Lijtime}
\quadLL_{ij} & = \phi_i \rho_j  \bc^T \bE^{-1} e^{\bA \bE^{-1} \tau_i} \bE e^{\bE^{-1}\bA t_j} \bE^{-1} \bb  = \phi_i \rho_j h(\tau_i+t_j),
\end{align}
where $h : \IR \rightarrow \IR$ is the impulse response of the underlying dynamical system, as
\begin{equation}\label{imp_resp_def}
h(t) = \bc^T e^{\bE^{-1}\bA t}\bE^{-1} \bb.
\end{equation}
The last equality in \eqref{Lijtime} follows from the observation 
\begin{equation}
\left(\bE^{-1} e^{\bA \bE^{-1} \tau_i}\right) \bE e^{\bE^{-1}\bA t_j}=\left( e^{\bE^{-1}\bA \tau_i}\bE^{-1}\right) \bE e^{\bE^{-1}\bA t_j} = e^{\bE^{-1}\bA (\tau_i+t_j)},
\end{equation}
which directly follows from inspection of the partial sums of the series expansion of the parenthesized expression. 
Using similar ideas, one can also show that  
\begin{equation}  \label{hijtime}
\quadLb = \quadL^T\bb = \left[ \begin{matrix}
\phi_1 \bc^T \bE^{-1} e^{\bA\bE^{-1} \tau_1} \bb
\\ \vdots \\ \phi_\nq \bc^T \bE^{-1} e^{\bA \bE^{-1} \tau_\nq} \bb
\end{matrix}  \right] =
\left[ \begin{matrix}
\phi_1 h(\tau_1) \\ 
\vdots \\ \phi_\nq h(\tau_\nq)
\end{matrix}  \right],
\end{equation}
and 
\begin{align}  \label{gijtime}
\quadcU^T= \bc^T\quadU &= \left[ \,
\rho_1  \bc^T e^{\bE^{-1}\bA t_1} \bE^{-1} \bb \quad 
\cdots \quad \bc^T  \rho_\np  e^{\bE^{-1}\bA t_\np} \bE^{-1} \bb \right] \nonumber \\
& = \left[ \,
\rho_1  h(t_1) \quad 
\cdots \quad \rho_\np h(t_\np)\right].
\end{align}
Now, let us consider the  $(i,j)$th component of $\quadMM$:
\begin{align}\label{Mijtime}
\quadMM_{ij} & = \phi_i \rho_j  \bc^T \bE^{-1} e^{\bA \bE^{-1} \tau_i} \bA e^{\bE^{-1}\bA t_j} \bE^{-1} \bb  = \phi_i \rho_j h'(\tau_i+t_j)
\end{align}
where 
\begin{equation}\label{der_imp_resp_def}
h'(t) = \frac{d}{dt}h(t) = \bc^T e^{\bE^{-1}\bA t}\bE^{-1}\bA \bE^{-1} \bb.
\end{equation}
As with the other quantities, $\quadMM$ can also be  obtained from the impulse response samples; yet in this case one will need to measure the derivative of $h'(t)$.

We have just shown that \textsf{Balanced POD} can be obtained directly from sampling the impulse response $h(t)$ and its derivative $h'(t)$ without access to state-snapshots data. A time-domain version of \QBT follows directly as given in Algorithm \ref{timequadbt}.

{
	\begin{algorithm}[htp] 
		\caption{Time-domain \QBT}
		\label{timequadbt}                            
		\algorithmicrequire~LTI system described through an impulse response evaluation map, $h(t)$; \\
		\hspace*{5mm} quadrature nodes, $t_j$, and weights, $\rho_j$, for $j=1,2,\ldots,\np$; \\
		\hspace*{5mm} quadrature nodes, $\tau_k$, and weights, $\phi_k$, for $k=1,2,\ldots,\nq$, \\
		\hspace*{5mm} and a truncation index, $1\leq r\leq \min(\np,\nq)$.
		
		\algorithmicensure~{Time-domain \QBT} reduced-system:  $\quadAr\in\mathbb{R}^{r \times r}, \ \quadbr, \quadcr \in \mathbb{R}^r$.
		
		\begin{algorithmic} [1]                                        
			\STATE \label{sampletime} Sample the impulse response 
			$\{h(t_j)\}_{j=1}^{\np}$  and $\{h(\tau_k)\}_{k=1}^{\nq}$ and its derivative 
			$\{h'(t_j)\}_{j=1}^{\np}$  and $\{h'(\tau_k)\}_{k=1}^{\nq}$.  Using the samples and quadrature weights, $\{\rho_j\}$ and $\{\phi_k\}$, 
			construct $\boldsymbol{\quadLL} \in \IC^{\nq \times \np}$, $\boldsymbol{\quadMM} \in \IC^{\nq \times \np}$, $\quadLb$, and $\quadcU$  as in \eqref{Lijtime}, \eqref{Mijtime}, \eqref{hijtime}, and \eqref{gijtime}, respectively, 
			\STATE Compute the SVD of $\quadLL$:
			\begin{equation}\label{SVD_BT2_time}
			\quadLL  = \left[ \begin{matrix}
			\quadZ_1 & \quadZ_{2}
			\end{matrix}  \right] \left[ \begin{matrix}
			\quadS_1 & \\ & \quadS_{2}
			\end{matrix}  \right] \left[ \begin{matrix}
			\quadY_1^* \\ \quadY_{2}^*
			\end{matrix}  \right],
			\end{equation}
			where $\quadS_1 \in \mathbb{R}^{r \times r}$ and $\quadS_{2} \in \mathbb{R}^{(\nq-r) \times (\np-r)}$.
			
			\STATE Construct the reduced order matrices:
			\begin{equation} \label{quadbtar_time}
			\begin{array}{cc} \quadEr= \quadS_1^{-1/2} \quadZ_1^* \ \quadLL \ \quadY_1   \quadS_1^{-1/2} = \bI_r, 
			\quad & \quad 
			\quadAr = \quadS_1^{-1/2} \quadZ_1^* \ \quadMM \ \quadY_1   \quadS_1^{-1/2}, \\[2mm]
			\quadbr =  \quadS_1^{-1/2} \quadZ_1^* \quadLb,\qquad \mbox{and} &\quad 
			\quadcr =  \quadcU^T \quadY_1   \quadS_1^{-1/2}.
			\end{array}
			\end{equation}
		\end{algorithmic}
	\end{algorithm}
}

\begin{remark}
	Opmeer in \cite{opmeer2011model} established a connection between \textsf{Balanced POD} and interpolatory model reduction when the state data from the simulations of \eqref{eq:primal} and \eqref{eq:dual}  obtained via a numerical ODE solver such as forward and backward Euler, and other multi-stage implicit methods. In
	\cite{BF20}, Bertram and Fa{\ss}bender have further expanded on this idea by giving explicit connections between the Butcher tableau of the underlying  Runge-Kutta method and the resulting interpolation points. 
\end{remark} 
\begin{remark}
	We note that  the impulse response measurements, $h(t_i)$ and  $h'(t_i)$, which are needed to construct the data matrices in \eqref{Lijtime}, \eqref{hijtime}, \eqref{gijtime}, and \eqref{Mijtime}, are assumed to be available. In practical scenarios, one would need to rely on the estimation of such values either from experimental data or from numerical simulation. We refer to, e.g., \cite{CompImpResp}, for various impulse response measurement approaches that are in use for acoustical systems. 
\end{remark}
\begin{remark}
	In related current work \cite{igor},  impulse response sampling of $h(t)$ and $h'(t)$  are employed to extend the Eigenvalue Realization Algorithm \cite{kung} to continuous-time dynamical systems.
\end{remark}

\subsection{Time-domain \QBT for discrete-time systems}  \label{timeQBTdisc}
For the (SISO) discrete-time dynamical system 
\begin{align} \label{discretesiso}
\begin{array}{c} \bE \bx_{k+1}= \bA\, \bx_{k} + \bb\, u_k, \\[1mm] y_k = \bc^T \bx_k,\end{array}
\end{align}
the time-domain formulation simplifies drastically due to the uniform discrete time-stepping.
The reachability Gramian $\bP$ for \eqref{discretesiso}, in  time-domain, is given by
\begin{align}
\bP = \sum_{k=0}^\infty (\bE^{-1}\bA)^k \bE^{-1}\bb \bb^T \bE^{-T} (\bA^T\bE^{-T})^k.
\end{align}
Time-domain \textsf{Balanced POD} for discrete-time systems \cite{ma2011reduced} uses 
\begin{align*}
\quadU = \begin{bmatrix}
\bE^{-1}\bb ~&~ (\bE^{-1}\bA) \bE^{-1}\bb ~&~ \cdots ~&~ (\bE^{-1}\bA)^{\np-1} \bE^{-1}\bb
\end{bmatrix} \in \IR^{n \times \np}
\end{align*}
as an approximate square-root factor for $\bP$, i.e., 
\begin{align*}
\bP \approx \quadP = \quadU \quadU^T = \sum_{k=0}^{\np-1} (\bE^{-1}\bA)^k \bE^{-1}\bb \bb^T \bE^{-T} (\bA^T\bE^{-T})^k.
\end{align*}
One may follow the same development for $\bQ$. For  \eqref{discretesiso}, $\bQ$ is given by
\begin{align}
\bQ = \sum_{k=0}^\infty (\bE^{-T}\bA^T)^k \bE^{-T}\bc \bc^T \bE^{-1} (\bA\bE^{-1})^k.
\end{align}
Then, an approximate square-root factor for $\bQ$ is 
\begin{align*}
\quadL^T = \begin{bmatrix}
\bc^T\bE^{-1} \\ \bc^T\bE^{-1} (\bA\bE^{-1}) \\ \vdots \\ \bc^T\bE^{-1} (\bA\bE^{-1})^{\nq-1} 
\end{bmatrix} \in \IR^{\nq \times n},
\end{align*}
such that 
\begin{align*}
\bQ \approx \quadQ = \quadL \quadL^T =  \sum_{k=0}^{\nq-1} (\bE^{-T}\bA^T)^k \bE^{-T}\bc \bc^T \bE^{-1} (\bA\bE^{-1})^k.
\end{align*}
\noindent Note that these truncated sums are also the main tool in obtaining low-rank approximations to $\bP$ and $\bQ$ (see, e.g., \cite{smith1968matrix,penzl1999cyclic,gugercin2003modified}). 
Then, \cite{ma2011reduced} observes that 
$\quadLL = \quadL^T \bE \quadU$ can be directly obtained from data: Direct computation shows that
$$
\quadLL_{ij} = \bc^T (\bE^{-1}\bA)^{i+j-2} \bE^{-1} \bb.
$$
Thus in time-domain \textsf{Balanced POD} for discrete-time systems, $\quadLL$ is the Hankel matrix of Markov parameters. Similarly, one may directly show that $\quadMM=\quadL^T \bE \quadU$ is the shifted Hankel matrix of Markov parameters, i.e.,
$$
\quadMM_{ij} = \bc^T (\bE^{-1}\bA)^{i+j-1} \bE^{-1} \bb.
$$
Likewise, $\quadL^T\bb$ and $\bc^T\quadU$ can be directly obtained from Markov parameters. Indeed, \cite{ma2011reduced} observed that for the special case of discrete-time dynamical systems, time-domain \textsf{Balanced POD} can be obtained from input-output data using  Markov parameters. For discrete-time dynamical systems, the Markov parameters are  samples of the impulse response function $h(k) = \bc^T (\bE^{-1}\bA)^{k}\bE^{-1}\bb$ for $k=0,1,2,\ldots$. (Note that \cite{ma2011reduced} also observed  that this formulation of time-domain \textsf{Balanced POD} is equivalent to the
Eigenvalue Realization Algorithm \cite{kung}.)
Our formulation in \S \ref{timeQBTcont} generalizes concepts from \cite{ma2011reduced} to continuous-time systems and unifies them under a larger umbrella of impulse-response sampling; our framework in \S \ref{timeQBTcont} goes farther in removing the requirement of access to state-space quantities.

\section{Conclusion}
\label{sec:conc}

We introduce here a novel reformulation of one of the most important system-theoretic model reduction tools in use today, \emph{balanced truncation}.
In its usual formulation, this method requires intrusive access to the original system realization. 
By contrast, our reformulation requires only system response data, either measured or computed.   
The central theme driving our development involves approximation of Gramian-related quantities by means of convergent numerical quadratures, permitting results as close as desired to those of classical \BT. 
We demonstrate the application of our approach to a multitude of test cases, for both SISO and MIMO systems, and for both continuous-time and discrete-time systems. 
We outline time-domain extensions, connecting our approach to existing methods as well as offering promising future themes.

\section*{Acknowledgements}
The work of Beattie and Gugercin 
was supported in parts by National Science Foundation under Grant No. DMS-1923221 and DMS-1819110.

\bibliographystyle{spmpsci}      
\bibliography{GGB20_BT_data_ref}   

\begin{thebibliography}{10}
\providecommand{\url}[1]{{#1}}
\providecommand{\urlprefix}{URL }
\expandafter\ifx\csname urlstyle\endcsname\relax
  \providecommand{\doi}[1]{DOI~\discretionary{}{}{}#1}\else
  \providecommand{\doi}{DOI~\discretionary{}{}{}\begingroup
  \urlstyle{rm}\Url}\fi

\bibitem{ACA05}
Antoulas, A.C.: Approximation of large-scale dynamical systems.
\newblock SIAM, Philadelphia (2005)

\bibitem{AntBG20}
Antoulas, A.C., Beattie, C., Gugercin, S.: Interpolatory methods for model
  reduction.
\newblock Computational Science and Engineering 21. SIAM, Philadelphia (2020)

\bibitem{morAntGI16}
Antoulas, A.C., Gosea, I.V., Ionita, A.C.: Model reduction of bilinear systems
  in the {L}oewner framework.
\newblock SIAM J. Scientific Computing \textbf{38}(5), B889--B916 (2016)

\bibitem{ALI17}
Antoulas, A.C., Lefteriu, S., Ionita, A.C.: A tutorial introduction to the
  {L}oewner framework for model reduction ({C}hap. 8).
\newblock In: Model Reduction and Approximation. SIAM (2017)

\bibitem{BBF14}
Baur, U., Benner, P., Feng, L.: Model order reduction for linear and nonlinear
  systems: A system-theoretic perspective.
\newblock Archives of Computational Methods in Engineering \textbf{21}(4),
  331--358 (2014)

\bibitem{BeattieGugercin2012RlznIndH2Approx}
Beattie, C., Gugercin, S.: Realization-independent
  $\mathcal{H}_2$-approximation.
\newblock In: 2012 IEEE 51st IEEE Conference on Decision and Control (CDC), pp.
  4953--4958. IEEE (2012)

\bibitem{damm11}
Benner, P., Damm, T.: {L}yapunov equations, energy functionals and model order
  reduction of bilinear and stochastic systems.
\newblock SIAM J. on Control and Optimization \textbf{49}, 686--711 (2011)

\bibitem{morBTQBgoyal}
Benner, P., Goyal, P.: Balanced truncation model order reduction for
  quadratic-bilinear control systems.
\newblock ar{X}iv preprint available at \url{https://arxiv.org/abs/1705.00160}
  (2017)

\bibitem{BKS14}
Benner, P., K\"urschner, P., Saak, J.: Self-generating and efficient shift
  parameters in {ADI} methods for large {L}yapunov and {S}ylvester equations.
\newblock Electronic Transactions on Numerical Analysis \textbf{43}, 142--162
  (2014)

\bibitem{siammorbook2017}
Benner, P., Ohlberger, M., Cohen, A., Willcox, K.: Model Reduction and
  Approximation.
\newblock SIAM, Philadelphia, PA (2017)

\bibitem{BenS13}
Benner, P., Saak, J.: Numerical solution of large and sparse continuous time
  algebraic matrix {Riccati} and {Lyapunov} equations: {A} state of the art
  survey.
\newblock GAMM Mitteilungen \textbf{36}(1), 32--52 (2013)

\bibitem{benner2010balanced}
Benner, P., Schneider, A.: Balanced truncation model order reduction for {LTI}
  systems with many inputs or outputs.
\newblock In: Proceedings of the 19th International Symposium on Mathematical
  Theory of Networks and Systems--MTNS, vol.~5 (2010)

\bibitem{BF19}
Bertram, C., Fa{\ss}bender, H.: {L}yapunov and {S}ylvester equations: A
  quadrature framework.
\newblock ar{X}iv preprint available at \url{https://arxiv.org/abs/1903.05383}
  (2020)

\bibitem{BF20}
Bertram, C., Fa{\ss}bender, H.: A link between {G}ramian-based model order
  reduction and moment matching.
\newblock In: Model Reduction of Complex Dynamical Systems, pp. 119--139.
  Springer (2021)

\bibitem{bhatia2013matrix}
Bhatia, R.: Matrix analysis, vol. 169.
\newblock Springer (2013)

\bibitem{boyd1982optimization}
Boyd, J.P.: The optimization of convergence for chebyshev polynomial methods in
  an unbounded domain.
\newblock Journal of computational physics \textbf{45}(1), 43--79 (1982)

\bibitem{Boyd87}
Boyd, J.P.: Exponentially convergent {F}ourier-{C}hebshev quadrature schemes on
  bounded and infinite intervals.
\newblock J. Sci. Comput \textbf{2}, 99--109 (1987)

\bibitem{breiten2016structure}
Breiten, T.: Structure-preserving model reduction for integro-differential
  equations.
\newblock SIAM J. Control and Optimization \textbf{54}(6), 2992--3015 (2016)

\bibitem{curtain1986}
Curtain, R.F., Glover, K.: Balanced realisations for infinite dimensional
  systems.
\newblock Operator Theory: Advances and Applications \textbf{19}, 87--104
  (1986)

\bibitem{DGB15}
Drma\v{c}, Z., Gugercin, S., Beattie, C.: Quadrature-based vector fitting for
  discretized $\mathcal{H}_2$ approximation.
\newblock SIAM J. Scientific Computing \textbf{37}(2), A625--A652 (2015)

\bibitem{drmac2015vector}
Drma\v{c}, Z., Gugercin, S., Beattie, C.: Vector fitting for matrix-valued
  rational approximation.
\newblock SIAM Journal on Scientific Computing \textbf{37}(5), A2346--A2379
  (2015)

\bibitem{DKS14}
Druskin, V., Knizhnerman, L., Simoncini, V.: Analysis of the rational {K}rylov
  subspace and {ADI} methods for solving the {L}yapunov equation.
\newblock SIAMMatrix \textbf{35}(2), 476--498 (2014)

\bibitem{morGosA18}
Gosea, I.V., Antoulas, A.C.: Data-driven model order reduction of
  quadratic-bilinear systems.
\newblock NumAlgAppl \textbf{25}(6), e2200 (2018)

\bibitem{morGosPAetal18}
Gosea, I.V., Petreczky, M., Antoulas, A.C., Fiter, C.: Balanced truncation for
  linear switched systems.
\newblock Advances in Computational Mathematics \textbf{44}(6), 1845--1886
  (2018)

\bibitem{gugercin2003modified}
Gugercin, S., Sorensen, D.C., Antoulas, A.C.: A modified low-rank {S}mith
  method for large-scale {L}yapunov equations.
\newblock Numerical Algorithms \textbf{32}(1), 27--55 (2003)

\bibitem{HMT11}
Halko, N., Martinsson, P.G., Tropp, J.A.: Finding structure with randomness:
  {P}robabilistic algorithms for constructing approximate matrix
  decompositions.
\newblock {SIAM} Rev. \textbf{53}(2), 217 -- 288 (2011)

\bibitem{higham2002accuracy}
Higham, N.J.: Accuracy and stability of numerical algorithms.
\newblock SIAM (2002)

\bibitem{himpe2018emgr}
Himpe, C.: \textsf{emgr}---{The Empirical {G}ramian Framework}.
\newblock Algorithms \textbf{11}(7), 91 (2018)

\bibitem{morHim21}
Himpe, C.: Comparing (empirical-{G}ramian-based) model order reduction
  algorithms.
\newblock In: P.~Benner, {et al.} (eds.) Model Reduction of Complex Dynamical
  Systems. Springer (2021)

\bibitem{morIonA14}
Ionita, A.C., Antoulas, A.C.: Data-driven parametrized model reduction in the
  {L}oewner framework.
\newblock SIAM J. Scientific Computing \textbf{36}(3), A984--A1007 (2014)

\bibitem{Juang94}
Juang, J.: Applied system identification.
\newblock Prentice Hall, Inc. (1994)

\bibitem{keysight}
{Keysight} Technologies, Santa Rosa, California, United States: S-Parameter
  Measurements: Basics for High Speed Digital Engineers.
\newblock
  \urlprefix\url{http://literature.cdn.keysight.com/litweb/pdf/5991-3736EN.pdf}

\bibitem{kung}
Kung, S.: A new identification and model reduction algorithm via singular value
  decomposition.
\newblock In: Proceedings of 12th Asilomar Conference on Circuits, Systems $\&$
  Computers, Pacific Grove, CA, 1978, pp. 705--714 (1978)

\bibitem{kurschner2016efficient}
K{\"u}rschner, P.: Efficient low-rank solution of large-scale matrix equations.
\newblock Ph.D. thesis, Otto von Guericke Universit\"{a}t Magdeburg (2016)

\bibitem{li2002low}
Li, J.R., White, J.: Low rank solution of {L}yapunov equations.
\newblock SIAM J. Matrix Analysis and Applications \textbf{24}(1), 260--280
  (2002)

\bibitem{ma2011reduced}
Ma, Z., Ahuja, S., Rowley, C.W.: Reduced-order models for control of fluids
  using the eigensystem realization algorithm.
\newblock Theoretical and Computational Fluid Dynamics \textbf{25}(1-4),
  233--247 (2011)

\bibitem{mayo2007fsg}
Mayo, A., Antoulas, A.: {A framework for the solution of the generalized
  realization problem}.
\newblock Linear Algebra and Its Applications \textbf{425}(2-3), 634--662
  (2007)

\bibitem{moore1981principal}
Moore, B.: Principal component analysis in linear systems: {C}ontrollability,
  observability, and model reduction.
\newblock IEEE Transactions on Automatic Control \textbf{26}(1), 17--32 (1981)

\bibitem{mullis1976synthesis}
Mullis, C., Roberts, R.: Synthesis of minimum roundoff noise fixed point
  digital filters.
\newblock Circuits and Systems, IEEE Transactions on \textbf{23}(9), 551--562
  (1976)

\bibitem{Niconet}
{Niconet}: Benchmark examples for model reduction of linear time invariant
  dynamical systems.
\newblock
  \urlprefix\url{http://slicot.org/20-site/126-benchmark-examples-for-model-reduction}

\bibitem{opmeer2011model}
Opmeer, M.R.: Model order reduction by balanced proper orthogonal decomposition
  and by rational interpolation.
\newblock IEEE Transactions on Automatic Control \textbf{57}(2), 472--477
  (2011)

\bibitem{morPehGW17}
Peherstorfer, B., Gugercin, S., Willcox, K.: Data-driven reduced model
  construction with time-domain {L}oewner models.
\newblock SIAM J. Scientific Computing \textbf{39}(5), A2152 -- A2178 (2017)

\bibitem{penzl1998numerical}
Penzl, T.: Numerical solution of generalized {L}yapunov equations.
\newblock Advances in Computational Mathematics \textbf{8}(1-2), 33--48 (1998)

\bibitem{penzl1999cyclic}
Penzl, T.: A cyclic low-rank smith method for large sparse {L}yapunov
  equations.
\newblock SIAM J. Scientific Computing \textbf{21}(4), 1401--1418 (1999)

\bibitem{phillips2004poor}
Phillips, J.R., Silveira, L.M.: Poor man's {TBR}: a simple model reduction
  scheme.
\newblock {IEEE} {T}rans. {C}omput. {A}ided {D}esign {I}ntegr. {C}ircuits
  {S}yst. \textbf{24}(1), 43--55 (2004)

\bibitem{igor}
Pontes~Duff, I.:  (2021).
\newblock Private communication.

\bibitem{quarteroni2015reduced}
Quarteroni, A., Manzoni, A., Negri, F.: Reduced basis methods for partial
  differential equations: an introduction, vol.~92.
\newblock Springer (2015)

\bibitem{reis2014balancing}
Reis, T., Selig, T.: Balancing transformations for infinite-dimensional systems
  with nuclear {H}ankel operator.
\newblock Integral Equations and Operator Theory \textbf{79}(1), 67--105 (2014)

\bibitem{Ro05}
Rowley, C.W.: Model reduction for fluids, using balanced proper orthogonal
  decomposition.
\newblock Int. J. Bifurcat. Chaos \textbf{15}(3), 997–--1013 (2005)

\bibitem{sabino2007solution}
Sabino, J.: Solution of large-scale {L}yapunov equations via the block modified
  {S}mith method.
\newblock Ph.D. thesis, Rice University (2007)

\bibitem{morSch93}
Scherpen, J.M.A.: Balancing for nonlinear systems.
\newblock Systems \& Control Letters \textbf{21}(2), 143--153 (1993)

\bibitem{schulze2018data}
Schulze, P., Unger, B., Beattie, C.A., Gugercin, S.: Data-driven structured
  realization.
\newblock Linear Algebra and its Applications \textbf{537}, 250--286 (2018)

\bibitem{simoncini2016computational}
Simoncini, V.: Computational methods for linear matrix equations.
\newblock SIAM Review \textbf{58}(3), 377--441 (2016)

\bibitem{singler2009proper}
Singler, J.R., Batten, B.A.: A proper orthogonal decomposition approach to
  approximate balanced truncation of infinite dimensional linear systems.
\newblock Intern. J. Computer Mathematics \textbf{86}(2), 355--371 (2009)

\bibitem{smith1968matrix}
Smith, R.: Matrix equation {$\mathbf{XA}+\mathbf{BX}=\mathbf{C}$}.
\newblock SIAM J. Appl. Math. \textbf{16}(1), 198--201 (1968)

\bibitem{CompImpResp}
Stan, G.B.V., Embrechts, J.J., Archambeau, D.: Comparison of different impulse
  response measurement techniques.
\newblock J. of the Audio Engineering Society \textbf{50}(4), 249--262 (2002)

\bibitem{Stykel08}
Stykel, T.: Low-rank iterative methods for projected generalized {L}yapunov
  equations.
\newblock Electron. Trans. Numer. Anal. (ETNA) \textbf{30}, 187--202 (2008)

\bibitem{sun1991perturbation}
Sun, J.G.: Perturbation bounds for the {C}holesky and {QR} factorizations.
\newblock BIT Numerical Mathematics \textbf{31}(2), 341--352 (1991)

\bibitem{trefethen2014exponentially}
Trefethen, L.N., Weideman, J.: The exponentially convergent trapezoidal rule.
\newblock SIAM Review \textbf{56}(3), 385--458 (2014)

\bibitem{weidemanTrefethen2007paraHyperContours}
Weideman, J., Trefethen, L.: Parabolic and hyperbolic contours for computing
  the {B}romwich integral.
\newblock Mathematics of Computation \textbf{76}(259), 1341--1356 (2007)

\bibitem{WP02}
Willcox, K., Peraire, J.: Balanced model reduction via the proper orthogonal
  decomposition.
\newblock AIAA Journal \textbf{40}(11), 2323--2330 (2002)

\bibitem{Vibrometry14}
Yang, S., Allen, M.S.: Harmonic transfer function to measure translational and
  rotational velocities with continuous-scan laser {D}oppler vibrometry.
\newblock J. Vibration and Acoustics \textbf{136}(2), 021025 (2014)

\end{thebibliography}

\subsection*{Appendix A: Two numerical quadrature rules}
\label{AppendixA}

Our reformulation of balanced truncation makes central use of numerical quadrature rules, and we choose two particular quadrature rules, the exponential trapezoid rule and a refinement of the Clenshaw–Curtis quadrature rule developed by Boyd, to illustrate the approach.  These strategies are effective in approximating matrix-valued integrals which generically appear as: $ \int_{-\infty}^{\infty} \bF(\zeta) d \zeta$ where
we assume that, for $\zeta\in\mathbb{R}$, the integrand $\bF(\zeta)$ is continuously differentiable, $\bF(-\zeta) = \bF(\zeta)^T$, and
$\lim_{\zeta\rightarrow \pm\infty} \zeta^2\bF(\zeta) =\bM=\bM^T\in\mathbb{R}^{n\times n}$ (so that, in particular, $\|\bF(\zeta)\|=\mathcal{O}\left(\frac{1}{\zeta^2}\right)$ as $\zeta\rightarrow \pm\infty$).  This assumed form  covers the Gramian integrals, \eqref{gram_P_freq} and \eqref{gram_Q}.  The quadrature rules that we describe below are typical of possible choices though certainly there will be many other refined approaches that one may consider (see especially \cite{trefethen2014exponentially}).

\textsf{[ExpTrap]: The exponential trapezoid rule.}  Note first that
{\small
	\begin{align*}
	\int_{-\infty}^{\infty} \bF(\zeta) d \zeta & =
	\int_{0}^{\infty} \bF(\zeta)+\bF(\zeta)^T\ d \zeta 
	=L \int_{-\infty}^{\infty} e^{L\tau}\left(\bF(e^{L\tau})+\bF(e^{L\tau})^T\right)\ d \tau =  \int_{-\infty}^{\infty} \bG(\tau)\; d \tau, 
	\end{align*}}\noindent where $\bG(\tau)=L\,e^{L\tau}\left(\bF(e^{L\tau})+\bF(e^{L\tau})^T\right)$; the first equality has used the symmetry of $\bF(\cdot)$; and the second equality has introduced a change-of-variable for $\zeta>0$ as $\zeta=e^{L\tau}$.  For any $h>0$ and $t\in\mathbb{R}$,  $\sum_{k=-\infty}^{\infty} \bG(t+khL)$, viewed as a function of $t$, is well-defined, continuously differentiable, and  $hL$-periodic. In particular, it has a pointwise convergent Fourier series with (matrix-valued) coefficients:
\begin{align*}
\sum_{k=-\infty}^{\infty} \bG(t+khL) = \frac{2\pi}{hL}\sum_{\ell=0}^{\infty} \widehat{\bG}\left(\frac{2\pi\ell}{hL}\right)\; e^{ \dot{\imath\!\imath}\,\frac{2\pi\ell}{hL}t}  
\end{align*}
where $\widehat{\bG}(\omega)$ denotes the Fourier transform of $\bG(t)$.  Evaluating at $t=0$, this expression may be rearranged to
$$
hL\sum_{k=-\infty}^{\infty} \bG(khL) -  \int_{-\infty}^{\infty} \bG(\tau)\; d \tau = 2\pi\sum_{\ell=1}^{\infty} \widehat{\bG}\left(\frac{2\pi\ell}{hL}\right)
$$
or equivalently,
\begin{equation}\label{quadDisc}
hL \sum_{k=-\infty}^{\infty} e^{khL}\left(\bF(e^{khL})+\bF(e^{khL})^T\right)-\int_{-\infty}^{\infty} \bF(\zeta)\; d \tau = 2\pi\sum_{\ell=1}^{\infty} \widehat{\bG}\left(\frac{2\pi\ell}{hL}\right). \tag{A.1}
\end{equation}
The right-hand side becomes exponentially close to zero as $h\rightarrow 0$.  The first sum on the left-hand side manifests the trapezoid rule, which becomes computationally tractable with suitable truncation of the sum: say, $\nu_0\leq k\leq \nu_0+\varpi$.
This leads to: 
$$
\int_{-\infty}^{\infty} \bF(\zeta) d \zeta  \approx hLe^{\nu_{0}hL}\bF(0) +  hLe^{\nu_{0}hL} \sum_{\ell=0}^{\varpi} e^{\ell hL}\left(\bF(e^{(\nu_0+\ell) hL})+\bF(-e^{(\nu_0+\ell) hL})\right),
$$
which includes a correction that weights the value of $\bF$ at $\zeta=0$ as well.  

The right-hand side of \eqref{quadDisc} gives the discretization error if the sum on the left-hand side were not truncated.   This discretization error is on the order of $e^{-\frac{2\pi a}{hL}}$ as $h\rightarrow 0$ where $a\approx \inf |\Re e(\lambda)|$  taken over all poles $\lambda$ of $\mathbf{A}$ (e.g., see \cite[Theorem 5.2]{trefethen2014exponentially}).  If we assume for simplicity that the left-hand sum in \eqref{quadDisc} is truncated symmetrically so that $\nu_0\approx -\frac{\varpi}{2}$ and note that $\|\bG(\tau)\|=\mathcal{O}(e^{-L|\tau|})$ as $\tau\rightarrow \pm \infty$, then balancing  discretization error with truncation error will suggest that the optimal number of quadrature nodes will grow proportionately with $\frac{1}{h^2}$ and the quadrature error will then decrease exponentially fast as the number of quadrature nodes, $\varpi$, increases:  roughly as $\mathcal{O}(e^{-\alpha\sqrt{\varpi}})$ with $\alpha\approx \sqrt{\pi a}$.

\textsf{[B/CC]: The Boyd/Clenshaw–Curtis rule.}  Boyd adapted the Clenshaw-Curtis quadrature rule to infinite integration domains in \cite{Boyd87}. We summarize the application of this rule to our setting.  
Starting with the change-of-variable $\zeta = L \cot(\tau)$, we find:
\begin{equation} \label{BCCcov}
\int_{-\infty}^{\infty} \bF(\zeta) d \zeta = \int_{0}^{\pi} \frac{\bF(L \cot(\tau)) L}{\sin^2(\tau)} d \tau.  \tag{A.2}
\end{equation}
The right-hand integrand is evidently $\pi$-periodic. With the asymptotics assumed for $\bF$ we have that limits exist for the right-hand integrand both at $0$ and $\pi$ and they are real, symmetric, and equal:
\begin{align*} 
\lim_{\tau\rightarrow 0} \frac{\bF(L \cot(\tau)) L}{\sin^2(\tau)} 
=& \frac{1}{L} \lim_{\tau\rightarrow 0} \left(L \cot(\tau)\right)^2\bF(L \cot(\tau))=
\frac{1}{L} \lim_{\zeta\rightarrow \infty} \zeta^2\bF(\zeta)= \frac{1}{L}\bM \\[2mm]
\lim_{\tau\rightarrow \pi} \frac{\bF(L \cot(\tau)) L}{\sin^2(\tau)} 
=& \frac{1}{L} \lim_{\tau\rightarrow \pi} \left(L \cot(\tau)\right)^2\bF(L \cot(\tau))=
\frac{1}{L} \lim_{\zeta\rightarrow -\infty} \zeta^2\bF(\zeta)= \frac{1}{L}\bM.
\end{align*} 
The trapezoid rule for the right-hand integral in \eqref{BCCcov} is a compelling choice.  Defining $h=\frac{\pi}{\varpi+1}$ for $\varpi\in\mathbb{N}$, take equally spaced points $\tau_{\ell}=\ell h $
for $\ell=0,1,...,\varpi+1$,  giving
\begin{equation} \label{BCCformula}
\int_{-\infty}^{\infty} \bF(\zeta) d \zeta  \approx \sum_{\ell=1}^{\varpi} hL
\frac{\bF(L \cot(\tau_{\ell}))}{\sin^2(\tau_{\ell})} + \frac{h}{L}\bM 
= \sum_{\ell=1}^{\varpi}  \rho_{\ell}^2 
\bF(\omega_{\ell}) + \rho_{\infty}^2\bM, \tag{A.3}
\end{equation}
where $\rho_{\ell}^2 = \frac{L\pi}{(\varpi+1)\sin^2(\tau_{\ell})}$ and $\omega_{\ell}=L \cot(\tau_{\ell}))$ 
for $\ell=1,...,\varpi$ and $\rho_{\infty}^2=\frac{\pi}{L(\varpi+1)}$. 
The last term in \eqref{BCCformula} is $\frac{h}{2}$ multiplying the combined evaluation of the periodicized integrand at the endpoints, $0$ and $\pi$, corresponding to the ``nodes at infinity'' for the original integral.  
We note that a similar derivation was described in \cite[\S3.3]{DGB15} for the computation of the $\mathcal{H}_2$ norm for linear dynamical systems. We refer the reader to 
\cite{boyd1982optimization} and \cite{drmac2015vector} for details guiding the choice of $L$ in practice.

\subsection*{Appendix B: Proof of Proposition \ref{prop:QuadError}} 
\label{AppendixB}

The hypotheses imply  that  
$$
\|\bQ-\quadQ\|_2<\sigma_{\mathsf{min}}(\bQ)\quad \mbox{and}\quad
\|\bP-\quadP\|_2<\sigma_{\mathsf{min}}(\bP).
$$  We then use results from  \cite[Thm.~1.4]{sun1991perturbation} (see also \cite[Thm.~10.8]{higham2002accuracy}) to build 
isometries $\boldsymbol{\Psi}_p\in\mathbb{C}^{\np \times n}$ and $\boldsymbol{\Psi}_q\in\mathbb{C}^{\nq \times n}$ and perturbations $\Delta\bL,\,\Delta\bU\in\mathbb{C}^{n \times n}$ such that 
$$
\quadU^*=\boldsymbol{\Psi}_p(\bU+\Delta\bU)^T   \quad \mbox{and} \quad\quadL^*=\boldsymbol{\Psi}_q(\bL+\Delta\bL)^T.
$$
Defining $\varepsilon_q=\frac{\|\bQ-\quadQ\|_F}{\|\bQ\|_2}$ and 
$\varepsilon_p=\frac{\|\bP-\quadP\|_F}{\|\bP\|_2}$, the hypotheses assure that both 
$$
\mathsf{cond}(\mathbf{Q})\varepsilon_q \leq \frac{\delta}{1+\delta}<1\quad \mbox{and}\quad
\mathsf{cond}(\mathbf{P})\varepsilon_p\leq \frac{\delta}{1+\delta}<1,
$$
and then,
referring again to \cite[Thm.~1.4]{sun1991perturbation},
we have further 
{\small
	$$
	\frac{\|\Delta\bU\|_F}{\|\bU\|_2}\leq \left(\frac{1}{\sqrt{2}}\right)\frac{\mathsf{cond}(\mathbf{P})\varepsilon_p}{1-\mathsf{cond}(\mathbf{P})\varepsilon_p}\leq \frac{\delta}{\sqrt{2}} \quad \mbox{and} \quad
	\frac{\|\Delta\bL\|_F}{\|\bL\|_2}\leq \left(\frac{1}{\sqrt{2}}\right)\frac{\mathsf{cond}(\mathbf{Q})\varepsilon_q}{1-\mathsf{cond}(\mathbf{Q})\varepsilon_q}\leq \frac{\delta}{\sqrt{2}}.
	$$
}
Now observe
{\small 
	\begin{align*}
	\quadL^* \bE \quadU &=\boldsymbol{\Psi}_q(\bL+\Delta\bL)^T\bE(\bU+\Delta\bU)\boldsymbol{\Psi}_p^T\\
	=&\boldsymbol{\Psi}_q(\bL^T\bE\bU)\boldsymbol{\Psi}_p^T
	+\boldsymbol{\Psi}_q\,(\Delta\bL)^T(\bE\bU)\boldsymbol{\Psi}_p^T  +\boldsymbol{\Psi}_q(\bL^T\bE)(\Delta\bU)\,\boldsymbol{\Psi}_p^T
	+\boldsymbol{\Psi}_q\,(\Delta\bL)^T\,\bE\,(\Delta\bU)\,\boldsymbol{\Psi}_p^T
	\end{align*}
}
and as a result, 
\begin{align*}
&\|\quadL^* \bE \quadU -\boldsymbol{\Psi}_q(\bL^T\bE\bU)\boldsymbol{\Psi}_p^T\|_F  \leq
\|\boldsymbol{\Psi}_q\,\Delta\bL^T(\bE\bU)\boldsymbol{\Psi}_p^T \|_F\quad  + \\ 
& \hspace*{4cm} \|\boldsymbol{\Psi}_q(\bL^T\bE)(\Delta\bU)\,\boldsymbol{\Psi}_p^T\|_F
+ \| \boldsymbol{\Psi}_q\,(\Delta\bL)^T\,\bE\,(\Delta\bU)\,\boldsymbol{\Psi}_p^T\|_F\\
& \hspace*{2cm}\leq  \|\Delta\bL\|_F\, \|\bE\bU\|_2  + \|\bL^T\bE\|_2\,\|\Delta\bU\|_F
+ \|\Delta\bL\|_F\,\|\bE\|_2 \|\Delta\bU\|_F \\
& \hspace*{2cm}\leq \frac{\delta}{\sqrt{2}}\, \|\bL\|_2\, \|\bE\bU\|_2  + \frac{\delta}{\sqrt{2}}\,\|\bL^T\bE\|_2\,\|\bU\|_2
+ \left(\frac{\delta}{\sqrt{2}}\,\right)^2\|\bL\|_2\,\|\bE\|_2 \|\bU\|_2 \\
& \hspace*{2cm}\leq \left(\sqrt{2}\,\delta+\frac{\delta^2}{2}\right) \|\bL\|_2\, \|\bE\|_2\,\|\bU\|_2 < 2\,\delta\, \|\bL\|_2\, \|\bE\|_2\,\|\bU\|_2,
\end{align*}
where the last inequality takes into account that $0<\delta<1$. 

For the final assertion, we note that the singular values of 
$\boldsymbol{\Psi}_q(\bL^T\bE\bU)\boldsymbol{\Psi}_p^*$ are
$\{\sigma_1,\, \sigma_2,\, \cdots,\, \sigma_n\}$ augmented with $0$.
A consequence of Lidskii's Majorization Theorem (e.g., \cite[Thm.III.4.4]{bhatia2013matrix}) gives
$$
\sum_{k=1}^n(\widetilde{\sigma}_k-\sigma_k)^2 \leq \|\quadL^* \bE \quadU-\boldsymbol{\Psi}_q(\bL^T\bE\bU)\boldsymbol{\Psi}_p^*\|_F^2.
$$
Then \eqref{quadBound}, the first assertion of Proposition \ref{prop:QuadError},  gives the conclusion.
\qquad $\blacksquare$

\end{document}